\documentclass[11pt,a4paper,twoside]{article}
\usepackage[utf8]{inputenc}
\usepackage[english, french]{babel}
\usepackage[T1]{fontenc}
\usepackage{ucs}
\usepackage{amsmath}
\usepackage{amsfonts}
\usepackage{amssymb}
\usepackage{amsthm}
\usepackage[all]{xy}
\usepackage{enumerate}
\usepackage{enumitem}
\usepackage{tikz}
\usepackage{tikz-cd}
\usepackage{pgfplots}
\usetikzlibrary{patterns}
\usetikzlibrary{matrix,arrows}
\usepackage{ifthen}
\usepackage{everypage}
\usepackage{verbatim} 
\usepackage{mathtools}
\usepackage{url}
\usepackage{hyperref}
\usepackage{cleveref}
\usepackage{mathrsfs}
\usepackage{indentfirst}
\usepackage{titlesec}
\usepackage{fancyhdr}
\usepackage{fancyhdr}

\titleformat{\section}{\normalfont\scshape\centering}{\thesection}{1em}{}
\titleformat{\subsection}[runin]
       {\normalfont\bfseries}
       {\thesubsection}
       {0.5em}
       {}
       [.]
\titleformat{\subsubsection}[runin]
       {\normalfont\itshape}
       {\thesubsection}
       {0.5em}
       {}
       [.]

\usepackage[top=3cm, bottom=3cm, left=3cm, right=3cm]{geometry}

\usepackage{thmtools, thm-restate}

\DeclareMathOperator{\gal}{Gal}

\DeclareMathOperator{\etale}{\acute{e}t}

\DeclareMathOperator{\Mor}{Mor}
\DeclareMathOperator{\spec}{Spec}
\DeclareMathOperator{\Res}{Res}
\DeclareMathOperator{\codim}{codim}
\DeclareMathOperator{\Div}{Div}
\DeclareMathOperator{\divcart}{div}
\DeclareMathOperator{\Pic}{Pic}

\DeclareMathOperator{\Der}{Der}
\DeclareMathOperator{\coh}{H}
\DeclareMathOperator{\Hom}{Hom}
\DeclareMathOperator{\Card}{Card}

\declaretheorem[name=Théorème,numberwithin=section]{thm}
\newtheorem{prop}[thm]{Proposition}
\newtheorem{lem}[thm]{Lemme}
\newtheorem{cor}[thm]{Corollaire}
\newtheorem{fait}[thm]{Fait}

\newtheorem{schol}[thm]{Scholie}

\theoremstyle{remark}

\theoremstyle{definition}
\newtheorem{defi}[thm]{Définition}

\newtheorem{quest}[thm]{Question}

\makeatletter
\def\blfootnote{\gdef\@thefnmark{}\@footnotetext}
\makeatother

\begin{document}

\blfootnote{\textit{Date}: \today}

\thispagestyle{empty}
\begin{center}
\textbf{\Large{\textsc{Pureté de l'approximation forte sur le corps des fonctions d'une courbe algébrique complexe}}}
\end{center}
\begin{center}
\textsc{Elyes Boughattas}
\end{center}

\selectlanguage{english}
\begin{abstract}
Over the function field of a complex algebraic curve, strong approximation off a non-empty finite set of places holds for the complement of a codimension $2$ closed subset in a homogeneous space under a semisimple algebraic group, and for the complement of a codimension $2$ closed subset in an affine smooth complete intersection of low degree.
\end{abstract}
\selectlanguage{french}
\begin{abstract}
Sur le corps des fonctions d'une courbe algébrique complexe, la propriété d'approxi\-mation forte hors d'un ensemble fini non vide de places vaut pour le complémentaire d'un fermé de codimension $2$ d'un espace homogène sous un groupe algébrique semi-simple et également pour le complémentaire d'un fermé de codimension $2$ d'une intersection complète affine lisse de bas degré.
\end{abstract}

\section{Introduction}

Étudier l'existence de points rationnels sur une variété (principe de Hasse), ainsi que leur répartition (densité des points rationnels, approximation faible, approximation forte) est un problème difficile de géométrie algébrique sur un corps non clos, dont la résolution dépend généralement de l'arithmétique du corps de base.\\

Sur un corps de nombres, il est bien connu que l'approximation faible est un invariant birationnel des variétés lisses. Pour l'approximation forte, la situation est plus délicate. En étendant un résultat de Kneser dans le cas des groupes algébriques \cite{MR184945}, Minchev démontre dans \cite{MR984929} qu'une variété vérifiant l'approximation forte est géométriquement simplement connexe, ce qui assure que l'approximation forte n'est pas un invariant birationnel des variétés lisses. Toutefois Cao, Xu \cite{MR3893760} et Wei \cite{MR4204537} ont montré que tout ouvert de l'espace affine $\mathbf{A}^n$ dont le complémentaire est de codimension au moins $2$ vérifie l'approximation forte hors d'une place. Wittenberg a ensuite interrogé dans \cite[Problem 6]{AIM2014}, puis dans son article de survol \cite[Question 2.11]{MR3821185}, un énoncé de pureté pour l'approximation forte qui est le suivant:

\begin{quest}\label{quest1}
Soit $k$ un corps de nombres, $S$ un ensemble fini de places de $k$ et $X$ une variété lisse sur $k$ vérifiant l'approximation forte hors de $S$. Un ouvert de $X$ dont le complémentaire est de codimension au moins $2$ vérifie-t-il également l'approximation forte hors de $S$?
\end{quest}
Par la suite, Cao, Liang et Xu \cite{MR4030258} ont apporté une réponse positive à cette question hors d'un ensemble fini non vide de places, pour les groupes semi-simples, simplement connexes et quasi-déployés. Cao et Huang \cite{MR4208898} apportent également une réponse positive pour certaines familles de groupes algébriques semi-simples simplement connexes non quasi-déployés.
\\

Si, maintenant, le corps de base est le corps des fonctions $K$ d'une courbe projective, lisse et irréductible $\Gamma$ sur le corps des complexes $\mathbf{C}$, le panorama est différent. On note~$\Omega_K$ l'ensemble des places de~$K$ qui, par critère valuatif de propreté, s'identifie aux points fermés~$\Gamma(\mathbf{C})$ de $\Gamma$. Lorsque~$v\in\Omega_K$ on note $K_v$ le complété de $K$ en $v$ et $\mathscr{O}_v$ son anneau des entiers. On dit alors que $X$ vérifie \textit{l'approximation faible} lorsque le plongement diagonal de $X(K)$ dans $\prod_{v\in\Omega_K}X(K_v)$ est d'image dense, où le produit est muni de la topologie produit. Lorsque $X$ est une $K$-variété, c'est-à-dire un $K$-schéma séparé de type fini, et~$S$ est un ensemble fini de places de $K$, on note $\mathscr{O}_{K,S}$ les entiers de $K$ hors de $S$ et on appelle $\mathscr{O}_{K,S}$-\textit{modèle} de~$X$ tout $\mathscr{O}_{K,S}$-schéma séparé de type fini $\mathscr{X}$ tel que $\mathscr{X}\otimes_{\mathscr{O}_{K,S}} K$ est isomorphe à $X$.

Pour définir l'approximation forte sur $K$, notons également $\mathbf{A}_K$ l'anneau des adèles de~$K$, qu'on définit comme l'ensemble des~$(x_v)\in\prod_{v\in\Omega_K}K_v$ tels que $x_v\in\mathscr{O}_v$ pour $v$ hors d'un ensemble fini de places de $K$. Cet anneau est muni d'une topologie via le système fondamental de voisinages de $0$ indexé par les ensembles finis $S\subset\Omega_K$ et donné par les~$\prod_{v\in S}U_v\times\prod_{v\in\Omega_K\backslash S}\mathscr{O}_v$ où pour $v\in S$ on désigne par $U_v$ un voisinage de $0$ dans $K_v$. De même, si $S\subset\Omega_K$ est fini, on note $\mathbf{A}_K^S$ l'anneau des adèles hors de $S$, défini comme la projection de $\mathbf{A}_K$ sur $\prod_{v\in\Omega_k\backslash S}K_v$ et muni de la topologie finale pour cette projection. Quand~$S\subset\Omega_K$ est fini, on peut munir $X(\mathbf{A}_K^S)$ de la topologie adélique comme suit: si~$S'\subset\Omega_K$ est fini contenant $S$ et $\mathscr{X}$ est un $\mathscr{O}_{K,S'}$-modèle de $X$, on a une identification canonique $X(\mathbf{A}_K^S)=\varinjlim_{S'\subset T\text{ fini}}\prod_{v\in T\backslash S}X(K_v)\times\prod_{v\in\Omega_K\backslash T}\mathscr{X}(\mathscr{O}_v)$ (voir par exemple \cite[\S2]{MR2985010}) et la topologie du terme de droite est indépendante du choix de $S'$ ainsi que du modèle $\mathscr{X}$. La définition suivante de l'approximation forte est alors centrale au sein de ce texte :
\begin{defi}
Soit $K$ le corps des fonctions d'une courbe algébrique complexe, $S$ un ensemble fini de places de $K$ et $X$ une variété sur $K$. On dit que $X$ vérifie \textit{l'approximation forte hors de $S$} lorsque l'application diagonale de $X(K)$ dans~$X(\mathbf{A}_K^S)$ est d'image dense dès que $X(K_v)\neq\emptyset$ pour tout $v\in S$.
\end{defi}

Lorsque $X$ est une $K$-variété lisse, un point rationnel de $X$ correspond à une section d'un modèle propre de $X$ au-dessus de~$\Gamma$:~ainsi, étudier l'arithmétique de $X$ est dans de nombreux cas facilité par les méthodes géométriques de déformation des courbes algébriques. Par exemple, si $X$ est rationnellement connexe, Graber, Harris et Starr montrent par de telles méthodes, dans \cite[Theorem~1.2]{MR1937199}, qu'elle admet un point rationnel, et le lieu des points rationnels est aussitôt dense pour la topologie de Zariski par \cite[Theorem 2.13]{MR1158625}. Hassett et Tschinkel \cite{MR2208420} conjecturent en fait que l'approximation faible vaut pour toute variété rationnellement connexe sur le corps des fonctions de $\Gamma$. Bien que cette conjecture soit hors de portée pour le moment, Colliot-Thélène et Gille \cite{MR2029865} ont dégagé une multitude de variétés rationnellement connexes jouissant de l'approximation faible: les espaces homogènes sous un groupe linéaire connexe, les fibrations en de tels espaces homogènes au-dessus d'une variété vérifiant l'approximation faible et les surfaces de del Pezzo de degré au moins $4$. De même, de Jong et Starr établissent dans \cite{dJS2006} un résultat profond d'approximation faible pour les intersections complètes lisses de bas degré dans $\mathbf{P}^n$. Enfin, Findley démontre également un résultat d'approximation faible pour des intersections complètes dans les grassmaniennes dans \cite{MR2941693} et Minocchieri établit dans \cite{MR4099637} l'approximation faible pour des intersections complètes de bas degré dans des espaces projectifs pondérés.

Pour ce qui est de l'étude des points entiers, les premiers résultats sont dus à Hassett et Tschinkel dans \cite{MR2403393} où ils démontrent que les points entiers sont denses pour la topologie de Zariski, pour certaines variétés log Fano. Les premiers résultats d'approximation forte sont quant à eux dus à Chen et Zhu qui démontrent dans \cite{zbMATH06921639} que l'approximation forte vaut hors d'une place pour les intersection complètes affines lisses et rationnellement connexes, sous condition sur le multi-degré. Le cas des espaces homogènes sous un groupe semi-simple a ensuite été obtenu par Colliot-Thélène dans \cite{MR3782219}: l'auteur démontre en outre que le groupe multiplicatif sur $\mathbf{C}(\mathbb{P}^1)$ ne vérifie l'approximation forte hors d'aucun ensemble fini de places. Il s'ensuit qu'en toute généralité, l'approximation forte n'est ni un invariant birationnel des variétés lisses, ni l'apanage de certaines variétés géométriquement simplement connexes. Il nous semble alors naturel d'interroger un analogue de la question~\ref{quest1}, ce qui soulève la définition suivante:

\begin{defi}
Soit $K$ le corps des fonctions d'une courbe algébrique complexe, $S$ un ensemble fini de places de $K$ et $X$ une variété sur $K$. On dit que $X$ vérifie \textit{la pureté de l'approximation forte hors de $S$} si tout complémentaire d'un fermé de codimension $2$ de $X$ vérifie l'approximation forte hors de $S$.
\end{defi}

L'analogue suivant de la question \ref{quest1} se pose alors naturellement:

\begin{quest}\label{questpure}
Si $K$ est le corps des fonctions d'une courbe algébrique complexe, $S$ un ensemble fini de places de $K$ et $X$ une variété lisse sur $K$. Si $X$ vérifie l'approximation  forte hors de $S$, vérifie-t-elle également la pureté de l'approximation forte hors de $S$?
\end{quest}

Par la présente note, nous apportons d'une part, à travers le théorème \ref{thpureh}, une réponse positive à cette question pour les espaces homogènes sous un groupe semi-simple, généralisant de la sorte le résultat principal de Colliot-Thélène dans \cite{MR3782219}. Un corollaire que nous en donnons est la pureté de l'approximation forte pour certaines quadriques affines lisses. D'autre part, nous montrons dans le théorème \ref{thpuricl} que la pureté de l'approximation forte hors d'une place vaut pour les intersections complètes affines lisses de bas degré, ce qui constitue une généralisation des résultats établis par Chen et Zhu dans \cite{zbMATH06921639}. Ces deux résultats suggèrent que la question \ref{questpure} pourrait admettre une réponse positive en toute généralité.\\

Le premier énoncé que nous obtenons est une version «pure» du théorème d'approximation forte de Colliot-Thélène \cite[Théorème 3.4]{MR3782219}, qui repose en partie sur le théorème d'existence de Riemann.

\begin{restatable}{thm}{thpureh}\label{thpureh}
Si $S$ est un ensemble fini non vide de places de $K$, alors tout $K$-espace homogène sous un groupe algébrique connexe semi-simple vérifie la pureté de l'approximation forte hors de $S$.
\end{restatable}

Une conséquence du théorème précédent est le corollaire \ref{logquad} suivant qui donne un résultat de pureté de l'approximation forte pour les quadriques affines lisses: 

\begin{restatable}{cor}{quadaffcor}\label{quadaff}
Soit $S$ un ensemble fini non vide de places de $K$, un entier naturel $n\geq 3$ et $q(x_1,\dots,x_n)$ une forme quadratique non dégénérée sur $K$. Si $\alpha \in K^{\times}$, la quadrique affine d'équation $q(x_1,\dots,x_n)=\alpha$ vérifie la pureté de l'approximation forte hors de $S$.
\end{restatable}

Le deuxième résultat est quant à lui une version «pure» du théorème d'approximation de Chen et Zhu \cite[Theorem 1.3]{zbMATH06921639}:

\begin{restatable}{thm}{thpuricl}\label{thpuricl}
Soit $S$ un ensemble fini non vide de places de $K$ et $X$ une intersection complète lisse dans $\mathbf{P}^n_K$ de multi-degré $(d_1,\dots,d_c)$ telle que $\sum_{i=1}^cd_i^2\leq n$. Si $D$ est une section hyperplane lisse de $X$, alors $X\backslash D$ vérifie la pureté de l'approximation forte hors de~$S$.
\end{restatable}

La section $2$ de notre article est dédiée à l'énoncé d'une méthode de fibration et d'une méthode de descente pour la pureté de l'approximation forte, qui sont respectivement des variantes d'une méthode de fibration énoncée dans \cite[Proposition 3.1]{MR3007293} et d'une méthode de descente énoncée dans \cite[Théorème 4.4]{MR3782219}. À la manière de \cite[Théorème~4.4]{MR3782219}, un ingrédient principal de notre méthode de descente est le théorème d'existence de Riemann, présent en filigrane dans la démonstration du lemme \ref{lemJLCT}.

La section $3$ contient la démonstration du théorème \ref{thpureh} avec, dans un premier temps, le cas particulier où $G$ est simplement connexe et les stabilisateurs de $X$ sont triviaux:~pour parvenir à ce premier cas, on combine la méthode de fibration établie à la section~$3$ au théorème de décomposition de Bruhat, ce qui permet de raisonner par dévissage à la manière de Cao, Liang et Xu dans la démonstration de \cite[Theorem~3.6]{MR4030258}. Le passage au cas général est alors permis par la méthode de descente qui est établi dans la section~$3$. En exploitant ces résultats, on obtient enfin la pureté de l'approximation forte hors d'une place pour les quadriques affines lisses et pour le complémentaire d'une quadrique projective lisse, sur le corps des fonctions d'une courbe algébrique complexe. Il s'agit respectivement des versions «pures» de \cite[Corollaire~4.1]{MR3782219} et de \cite[Corollaire~4.5]{MR3782219}.

Le section $4$ est quant à elle dédiée au théorème \ref{thpuricl}. Pour ce faire, nous établissons dans un premier temps le théorème \ref{purafmod} qui est un critère de pureté pour l'approximation forte, qu'on combine ensuite dans la sous-section \ref{total} au résultat principal de Chen et Zhu \cite[Theorem~1.8]{zbMATH06921639} pour obtenir le résultat de pureté souhaité du théorème \ref{thpuricl}. Notre démonstration de ce critère de pureté repose de façon décisive sur une théorie de la déformation que nous développons dans les sous-sections \ref{algcomm}, \ref{espacetgt} et \ref{lemdeform}, et dont l'aboutissement est le lemme~\ref{lemkollar}, ainsi que son corollaire \ref{corcrucial}. De façon remarquable, notre critère de pureté donné dans le théorème \ref{purafmod} repose sur l'espace des coniques contenues dans l'intersection complète lisse considérée, là où Chen et Zhu énoncent un critère dans \cite[Theorem~1.6]{zbMATH06921639} faisant intervenir des courbes de degré quelconque, et l'appliquent ensuite dans \cite[Theorem~1.8]{zbMATH06921639} à l'espace des coniques.


\subsection*{Notations}

Au cours de ce texte, on fixe $K$ le corps des fonctions d'une courbe projective, lisse et irréductible~$\Gamma$ sur le corps des complexes $\mathbf{C}$. Notons $\Omega_K$ l'ensemble de ses places qui, par critère valuatif de propreté, s'identifie aux points fermés $\Gamma(\mathbf{C})$ de $\Gamma$. Lorsque $v\in\Omega_K$ on note $K_v$ le complété de $K$ en $v$: il est isomorphe à $\mathbf{C}((t))$ par \cite[Chapitre II, \S4, Théorème 2]{zbMATH03657913} et cet isomorphisme identifie son anneau des entiers $\mathscr{O}_v$ à $\mathbf{C} [[t]]$.

Lorsque $X$ est une $K$-variété, c'est-à-dire un $K$-schéma séparé de type fini, on note~$X^{(1)}$ l'ensemble de ses points de codimension $1$. Le fibré tangent de $X$ est noté $\mathscr{T}_X$ et si $X$ est lisse et $Y$ est une sous-variété lisse de $X$ on note $\mathscr{N}_{Y/X}$ le fibré normal de $Y$ dans~$X$.

Dès que $G$ est un $K$-groupe algébrique, on note $\coh^1(K,G)$ l'ensemble pointé de cohomologie galoisienne défini dans \cite[Chapitre I, \S5.1]{MR1324577}. D'après \cite[Chapitre I, \S5.2]{MR1324577}, il classifie les $K$-torseurs à gauche sous $G$, c'est-à-dire les $G$-variétés à gauche géométriquement isomorphes à $G$ comme $G$-variétés, à isomorphisme près de $G$-variétés. Lorsque~$G$ est linéaire, $T$ est un $K$-torseur à gauche sous~$G$ et~$Y$ est une $G$-variété à droite qui est quasi-projective, on note~$Y^T$ la variété obtenue comme quotient de $Y\times_KT$ par l'action de~$G$ définie par~$g.(y,t)=(yg^{-1},gt)$ (l'existence d'une telle variété étant par exemple assurée par \cite[Lemma~2.2.3]{MR1845760}).

Si $k$ est un corps, $Y$ une variété projective et $X$ une variété quasi-projective sur $k$, on note $\Mor(Y,X)$ le $k$-schéma localement noethérien paramétrant les morphismes de $k$-schémas de $Y$ dans $X$ (on pourra se référer à \cite[\S4.c.]{MR1611822}): il représente le foncteur de la catégorie des $k$-schémas dans celle des ensembles, envoyant un $k$-schéma $T$ sur l'ensemble des morphismes de $k$-schémas $Y\times_kT\rightarrow X$.

Au cours de ce texte, on dit également qu'une suite
\begin{center}
\begin{tikzcd}
A\arrow[r, "\iota"] & B\arrow[r,"p"] & C
\end{tikzcd}
\end{center}
d'ensembles, où $C$ est pointé par un élément $c$, est exacte lorsque $\iota(A)=p^{-1}(\left\lbrace c\right\rbrace)$. 

\section{Méthodes de fibration et de descente pour la pureté}

\subsection{Une méthode de fibration pour la pureté}

On établit ici la méthode de fibration pour la pureté de l'approximation forte suivante:

\begin{prop}\label{purmethfib}
Soit $f:Y\rightarrow X$ un morphisme lisse de $K$-variétés lisses géométriquement intègres, dont les fibres géométriques sont intègres. Soit $W$ un ouvert non vide de $X$ et $S\subset\Omega_K$ fini tels que:
\begin{enumerate}[label=(\roman*)]
\item la variété $X$ vérifie la pureté de l'approximation forte hors de $S$;
\item pour tout $w\in W(K)$, la $K$-variété $f^{-1}(w)$ vérifie la pureté de l'approximation forte hors de $S$;
\item si $v\in S$, l'application $f^{-1}(W)(K_v)\rightarrow W(K_v)$ est surjective.
\end{enumerate}
Alors $Y$ vérifie la pureté de l'approximation forte hors de $S$.
\end{prop}

Avant d'en donner une démonstration, commençons par énoncer une variante de la méthode de fibration \cite[Proposition 3.1]{MR3007293} sur le corps $K$, que nous utiliserons à plusieurs reprises:

\begin{prop}\label{methfibr}
Soit $f:Y\rightarrow X$ un morphisme lisse de $K$-variétés lisses géométriquement intègres, dont les fibres géométriques sont intègres. Soit $W$ un ouvert non vide de $X$ et $S\subset\Omega_K$ fini tels que:
\begin{enumerate}[label=(\roman*)]
\item les fibres de $f$ au-dessus des $K$-points de $W$ vérifient l'approximation forte hors de~$S$;
\item si $v\in S$, l'application $f^{-1}(W)(K_v)\rightarrow W(K_v)$ est surjective.
\end{enumerate}
Alors, $Y$ vérifie l'approximation forte hors de $S$ si et seulement si $X$ vérifie l'approximation forte hors de $S$.
\end{prop}
\begin{proof}
Commençons par choisir $T\subset\Omega_K$ fini contenant $S$ et $F:\mathscr{Y}\rightarrow\mathscr{X}$ un~$\mathscr{O}_{K,T}$-modèle lisse de $f$ où $\mathscr{Y}$ et $\mathscr{X}$ sont des $\mathscr{O}_{K,T}$-schémas lisses séparés, tels que:
\begin{enumerate}[label=(\alph*)]
\item les fibres géométriques de $F$ sont géométriquement intègres;
\item si $v\not\in T$, l'application $\mathscr{Y}(\mathscr{O}_v)\rightarrow\mathscr{X}(\mathscr{O}_v)$ est surjective.
\end{enumerate}
En effet, le point (a) résulte de \cite[\S9.7]{MR217086}; pour montrer (b), il suffit de vérifier que le morphisme $F$ donné par (a) induit une application surjective au niveau des $\kappa(v)$-points, ce qui est vrai car $\kappa(v)=\mathbf{C}$ est algébriquement clos: le point (b) s'en déduit via le lemme de Hensel.\\

Commençons par supposer que $X$ vérifie l'approximation forte hors de $S$ et montrons que $Y$ vérifie également l'approximation forte hors de $S$. Il s'agit de montrer que si $U_v\subset Y(K_v)$ est un ouvert non vide ($v\in T\backslash S$) et si $Y(K_v)\neq\emptyset$ ($v\in S$), alors $Y(K)$ rencontre $$\displaystyle\prod_{v\in S}Y(K_v)\times\prod_{v\in T\backslash S}U_v\times\prod_{v\not\in T}\mathscr{Y}(\mathscr{O}_v).$$ On peut d'emblée supposer que $U_v\subset f^{-1}(W)$: alors, puisque $f$ est lisse, $f(U_v)\subset W$ est un ouvert de $X(K_v)$ (car $K_v$ est hensélien). Par approximation forte sur $X$ hors de $S$, il existe $N\in X(K)$ tel que $$N\in\prod_{v\in S}X(K_v)\times\prod_{v\in T\backslash S}f(U_v)\times\prod_{v\not\in T}\mathscr{X}(\mathscr{O}_v)$$si bien que N provient d'un $\mathscr{O}_{K,T}$-point $\mathscr{N}$ de $\mathscr{X}$. Posons $Z=f^{-1}(N)$, de sorte que $\mathscr{Z}=F^{-1}(\mathscr{N})$ est un $\mathscr{O}_{K,T}$-modèle de $Z$.

Si $v\not\in T$, il suit alors de (b) que $\mathscr{Z}(\mathscr{O}_v)\neq\emptyset$ et si $v\in S$, il suit de (ii) que $Z(K_v)\neq\emptyset$. Enfin, si $v\in T\backslash S$ alors $U_v\cap Z(K_v)$ est un ouvert non vide de $Z(K_v)$. Le point (i) assure donc que $$\displaystyle\prod_{v\in S} Z(K_v)\times\prod_{v\in T\backslash S} (U_v\cap Z(K_v))\times\prod_{v\not\in T}\mathscr{Z}(\mathscr{O}_v)$$ rencontre $Z(K)$, ce qui conclut.\\

Supposons maintenant que $Y$ vérifie l'approximation forte hors de $S$ et montrons que $X$ vérifie également l'approximation forte hors de $S$. Il s'agit de montrer que si $V_v\subset X(K_v)$ est un ouvert non vide ($v\in T\backslash S$) et si $X(K_v)\neq\emptyset$ ($v\in S$), alors $X(K)$ rencontre $$\displaystyle\prod_{v\in S}X(K_v)\times\prod_{v\in T\backslash S}V_v\times\prod_{v\not\in T}\mathscr{X}(\mathscr{O}_v)$$
et pour ce faire, il suffit de démontrer que $Y(K)$ rencontre
$$\displaystyle\prod_{v\in S}Y(K_v)\times\prod_{v\in T\backslash S}f^{-1}(V_v)\times\prod_{v\not\in T}\mathscr{Y}(\mathscr{O}_v).$$
Or, pour $v\in S$ on a $X(K_v)\neq\emptyset$, donc le théorème des fonctions implicites assure que~$W(K_v)\neq\emptyset$. On déduit alors de (ii) que $Y(K_v)\neq\emptyset$ et on conclut en appliquant l'approximation forte sur $Y$ hors de $S$.
\end{proof}

Pour notre démonstration de la proposition \ref{purmethfib}, le lemme géométrique suivant est bien utile et une démonstration en est par exemple donnée dans \cite[Proposition 3.5]{MR4030258} :

\begin{lem}\label{dimfibre}
Soit $f:Y\rightarrow X$ un morphisme fidèlement plat entre variétés géométriquement intègres sur un corps $k$ et $c>0$ un entier. Si $U\subset Y$ est un ouvert vérifiant $\codim(Y\backslash U,Y)\geq c$, alors il existe un ouvert dense $V\subset X$ tel que pour tout $x\in V(\overline{k})$ on ait $$\codim(f^{-1}(x)\backslash (U\cap f^{-1}(X)),f^{-1}(x))\geq c.$$
\end{lem}

On en déduit aussitôt une démonstration de la méthode de fibration pour la pureté de l'approximation forte:

\begin{proof}[Démonstration de la proposition \ref{purmethfib}]
Soit $U$ un ouvert de $Y$ contenant $Y^{(1)}$. Suivant le lemme \ref{dimfibre}, il existe un ouvert dense $V$ de $X$ tel que pour tout $v\in V(K)$ on ait $f^{-1}(v)^{(1)}\subset f^{-1}(v)\cap U$. Il s'agit alors d'appliquer la proposition \ref{methfibr} au morphisme $g:U\rightarrow f(U)$ induit par $f$, et à l'ouvert $W'=W\cap V\cap f(U)$.

D'une part, $f(U)$ vérifie l'approximation forte hors de $S$ par hypothèse de pureté sur~$X$, car $f$ étant fidèlement plat, le fait que $Y^{(1)}\subset U$ assure que $X^{(1)}\subset f(U)$.

De plus, si $v\in W'(K)$, le choix de $V$ assure que la fibre $g^{-1}(v)$ est un ouvert de $f^{-1}(v)$ contenant $f^{-1}(v)^{(1)}$. Il suit par hypothèse que $g^{-1}(v)$ vérifie l'approximation forte hors de~$S$, ce qui montre que le point (i) de la proposition \ref{methfibr} est vérifié.

Enfin, le point (ii) de la proposition \ref{methfibr} est vérifié. En effet, si $v\in S$ et $w'\in W'(K_v)$, l'ouvert $g^{-1}(w')(K_v)$ est dense dans $f^{-1}(w')(K_v)$ par lisseté de $f$ et théorème des fonctions implicites: en particulier il est non vide, ce qui conclut.
\end{proof}

\subsection{Une méthode de descente pour la pureté}

Dans cette sous-section, on démontre une méthode de descente pour la pureté de l'approximation forte, qui est une variante de \cite[Théorème 4.4]{MR3782219}.

\begin{prop}\label{purdesc}
Soit $S$ un ensemble fini non vide de places de $K$, $G$ un $K$-groupe linéaire,~$X$ une $K$-variété et $Y \xrightarrow{f} X$ un $G$-torseur à droite. Si pour tout $\left[\xi\right]\in \coh^1(K,G)$ le tordu~$Y^{\xi}$ vérifie la pureté de l'approximation forte hors de $S$, avec en outre $Y^\xi(K_v)\neq\emptyset$ pour $v\in S$, alors $X$ vérifie la pureté de l'approximation forte hors de $S$.
\end{prop}

Cette méthode de descente utilise le lemme suivant, dont la démonstration, qui repose sur le théorème d'existence de Riemann, suit en filigrane celle de \cite[Théorème~2.2]{MR2029865}. Avant de l'énoncer, introduisons la notation suivante: si $(E_i, e_i)_{i\in I}$ est une famille d'ensembles pointés, on note $\prod_{i\in I}^{'}E_i$ l'ensemble des $(x_i)\in\prod_{i\in I}E_i$ tels que $x_i=e_i$ pour un ensemble cofini d'indices $i$, pointé en $(e_i)$.

\begin{lem}\label{lemJLCT}\cite[Théorème 3.3]{MR3782219}
Soit $S$ un ensemble fini non vide de places de $K$ et~$G$ un $K$-groupe linéaire. Alors l'application diagonale$$\coh^1(K,G)\rightarrow\sideset{}{'}\prod_{v\not\in S}\coh^1(K_v,G)$$est surjective.
\end{lem}

\begin{proof}[Démonstration de la proposition \ref{purdesc}]
Considérons $U$ un ouvert de $X$ contenant les points de codimension $1$. On dispose des suites exactes suivantes qui s'insèrent dans les lignes du diagramme commutatif:
\begin{equation}
\begin{tikzcd}
Y(K)\arrow[r]\arrow[d]& X(K)\arrow[r]\arrow[d] & \coh^1(K,G)\arrow[d," \psi"] \\
Y(\mathbf{A}_K^S) \arrow[r]& X(\mathbf{A}_K^S) \arrow[r,"\varphi"]& \prod_{v\not\in S}^{'}\coh^1(K_v,G)
\end{tikzcd}
\end{equation}
où la flèche $\psi$ est surjective d'après le lemme \ref{lemJLCT}. Considérons $\alpha\in U(\mathbf{A}_K^S)\subset X(\mathbf{A}_K^S)$: la surjectivité de $\psi$ donne $\sigma=\left[\xi\right]\in \coh^1(K,G)$ tel que $\psi(\sigma)=\varphi(\alpha)$. Considérons maintenant le diagramme commutatif suivant, où les lignes sont exactes:
\begin{equation}\label{secex}
\begin{tikzcd}
Y^{\xi}(K)\arrow[r]\arrow[d]& X(K)\arrow[r]\arrow[d] & \coh^1(K,G^{\xi})\arrow[d," \psi'"] \\
Y^{\xi}(\mathbf{A}_K^S) \arrow[r]& X(\mathbf{A}_K^S) \arrow[r,"\varphi'"]& \prod_{v\not\in S}^{'}\coh^1(K_v,G^{\xi})
\end{tikzcd}.
\end{equation}
Le choix de $\sigma$ assure que $\varphi'(\alpha)=\left\lbrace * \right\rbrace$. Puisque $f^{\xi}$ est fidèlement plat, l'ouvert $V={\left(f^{\xi}\right)}^{-1}(U)$ contient les points de codimension $1$ de $Y^{\xi}$. L'exactitude de la ligne du bas de~(\ref{secex}) assure que $\alpha$ provient de $\beta\in  V(\mathbf{A}_K^S)$. Puisque $V$ contient les points de codimension~$1$ de $Y^{\xi}$, il vérifie l'approximation forte hors de $S$. Il s'ensuit que $\beta$ peut être approché par un point rationnel $\zeta\in V(K)$: l'image de $\zeta$ dans $U(K)$ approche alors $\alpha$ dans $U(\mathbf{A}_K^S)$, ce qui conclut.
\end{proof}

\section{Espaces homogènes sous un groupe semi-simple}

Cette section est dédiée à la démonstration du théorème \ref{thpureh} dont on rappelle l'énoncé:

\thpureh*
\subsection{Résultats préliminaires}

Sur un corps de nombres, Cao et Xu \cite[Proposition~$3.6$]{MR3893760} et Wei \cite[Lemma 2.1]{MR4204537} ont montré qu'on a pureté de l'approximation forte pour les espaces affines. Le résultat vaut encore sur le corps des fonctions d'une courbe algébrique complexe:

\begin{prop}\label{affAf}
Si $S$ est un ensemble fini non vide de places de $K$, l'espace affine $\mathbf{A}^n_K$ vérifie la pureté de l'approximation forte hors de $S$.
\end{prop}
\begin{proof}
Commençons par remarquer que l'approximation forte hors de $S$ vaut pour~$\mathbf{A}_K^1$, ce qui correspond à l'énoncé d'algèbre commutative \cite[Theorem 4.11]{MR1972204}.

On montre ensuite par récurrence sur $n$ le résultat souhaité, le cas $n=1$ étant acquis. Supposons $n\geq 2$ et considérons $p:\mathbf{A}^n_K\rightarrow\mathbf{A}^1_K$ la projection suivant la dernière coordonnée et appliquons-lui la méthode de fibration pour la pureté de la proposition \ref{purmethfib}, en posant $W=\mathbf{A}^1_K$. La variété $\mathbf{A}^1_K$ vérifie l'approximation forte d'après le cas $n=1$, donc le point~(i) de la proposition \ref{purmethfib} est vérifié. Le point (ii) de la proposition \ref{purmethfib} est également vérifié par hypothèse de récurrence, car pour $x\in\mathbf{A}^1_K(K)$ on a $p^{-1}(x)=\mathbf{A}^{n-1}_K$. Enfin, le point (iii) de la proposition \ref{purmethfib} est vérifié car la projection $p(K_v):\mathbf{A}^n_K(K_v)\rightarrow\mathbf{A}^1_K(K_v)$ est clairement surjective.
\end{proof}

\begin{defi}
Soit $A$ une algèbre étale non nulle sur un corps $k$. L'unique sous-variété ouverte minimale $R$ de $\Res_{A/k}(\mathbf{G}_a)$ contenant le tore $\Res_{A/k}(\mathbf{G}_m)$, stable sous l'action de~$\Res_{A/k}(\mathbf{G}_m)$ sur $\Res_{A/k}(\mathbf{G}_a)$ et vérifiant $$\codim(\Res_{A/k}(\mathbf{G}_a)\backslash R,\Res_{A/k}(\mathbf{G}_a))\geq 2$$ est appelée \textit{la sous-variété torique standard} de $\Res_{A/k}(\mathbf{G}_a)$. Plus précisément, $R\otimes_k\overline{k}$ est le complémentaire de la réunion de tous les sous-espaces vectoriels de codimension $2$ dans~$\Res_{A/k}(\mathbf{G}_a)\otimes_k\overline{k}$.
\end{defi}

\begin{cor}\label{AFtor}
Soit $A$ une $K$-algèbre étale non nulle et~$R$ la sous-variété torique standard de $\Res_{A/K}(\mathbf{G}_a)$. Alors $R$ vérifie la pureté de l'approximation forte hors de tout ensemble fini non vide de places de $K$.
\end{cor}
\begin{proof}
Cela suit de la proposition \ref{affAf} et du fait que $\Res_{A/K}(\mathbf{A}^1_A)\simeq\mathbf{A}^d_K$, où $d=\dim_KA$.
\end{proof}

\subsection{Cas des groupes semi-simples simplement connexes}

Cette sous-section est dédiée à la démonstration du théorème suivant, établi par Cao, Liang et Xu \cite{MR4030258} dans le cas d'un corps de nombres lorsque le groupe est quasi-déployé, cette hypothèse étant automatiquement satisfaite sur le corps $K$:

\begin{thm}\label{thpurgrp}
Si $S$ désigne un ensemble fini non vide de places de $K$ et $G$ un $K$-groupe semi-simple, simplement connexe, alors $G$ vérifie la pureté de l'approximation forte hors de $S$.
\end{thm}

En vue de démontrer le théorème \ref{thpurgrp}, nous avons besoin de quelques résultats liminaires sur les groupes semi-simples simplement connexes quasi-déployés sur un corps de caractéristique nulle. Considérons $k$ un corps de caractéristique nulle et $G$ un $k$-groupe semi-simple simplement connexe quasi-déployé.

Soit $B$ un sous-groupe de Borel de $G$ défini sur $k$. D'après \cite[Theorem 16.33.(c)]{MR3729270}, on dispose d'une suite exacte courte
\begin{center}
\begin{tikzcd}
1\arrow[r]& B_u\arrow[r]& B\arrow[r, "p"]&T\arrow[r]& 1
\end{tikzcd}
\end{center}
où $B_u$ est le radical unipotent de $B$ et $T$ est un tore. Cette suite est par ailleurs scindée par \cite[Exposé XVII, Théorème 5.1.1.(i)(b)]{zbMATH03333085}: ceci permet d'identifier $T$ à un sous-groupe de $B$, si bien que $p$ est $T$-équivariante. Par \cite[Theorem 21.84]{MR3729270}, il existe un unique sous-groupe de Borel $B'$ de $G$ tel que $B\cap B'= T$, qui n'est autre que le sous-groupe de Borel opposé à $B$. En notant $B'_u$ le radical unipotent de~$B'$, l'application 
\begin{equation}\label{defphi0}
\begin{tikzcd}[row sep=0.4em, column sep=small]
B\times_k B'_u\arrow[r]& G\\
(b,u)\arrow[r, mapsto] & b.u
\end{tikzcd}
\end{equation} 
est une immersion ouverte d'image notée $V_0$ par \cite[Theorem 21.84]{MR3729270}, décrite comme la cellule ouverte de Bruhat. On peut alors définir un morphisme $T$-équivariant $\phi_0:V_0\rightarrow T$ faisant commuter le diagramme
\begin{equation}\label{defphi}
\begin{tikzcd}
V_0\arrow[r,"\phi_0"]&T \\
B\times_kB'_u\arrow[u, "\wr"]\arrow[r, "p_B"]& B\arrow[u, swap, "p"]
\end{tikzcd}
\end{equation}
où $p_B$ est la projection suivant $B$ et $\overline{k}\left[T\right]^{\times}/\overline{k}^{\times}\xrightarrow{\sim}\overline{k}\left[V_0\right]^{\times}/\overline{k}^{\times}$ est un isomorphisme de $\gal(\overline{k}/k)$-modules.\\

Rappelons qu'un tore $T$ est dit \textit{quasi-trivial} s'il est isomorphe à $\Res_{A/k}(\mathbf{G}_m)$ pour $A$ une $k$-algèbre étale. Lorsque $Z$ est un fermé d'une variété $X$, on note $\Div_Z(X)$ le groupe des diviseurs de Cartier de $X$ à support dans $Z$. On a alors le lemme suivant:

\begin{lem}\label{lemC16}
Le morphisme $\overline{k}\left[V_0\right]^{\times}/\overline{k}^{\times}\rightarrow\Div_{G_{\overline{k}}\backslash V_{0,\overline{k}}}(G_{\overline{k}})$ défini par $f\mapsto \divcart(f)$ est un isomorphisme de $\gal(\overline{k}/k)$-modules. En particulier, $T$ est quasi-trivial.
\end{lem}
\begin{proof}
On reprend l'argument de \cite[Lemma 3.1]{MR4030258}. Puisque $G$ est semi-simple, le lemme de Rosenlicht assure que $\overline{k}\left[G\right]^{\times}=\overline{k}^{\times}$. De plus, $G_{\overline{k}}$ étant semi-simple et simplement connexe, on a $\Pic(G_{\overline{k}})=0$ (voir \cite[\S4.3, Theorem 1]{MR1634406}). On en déduit aussitôt la première assertion. La deuxième s'en déduit du fait que $\overline{k}\left[T\right]^{\times}/\overline{k}^{\times}\xrightarrow{\sim}\overline{k}\left[V_0\right]^{\times}/\overline{k}^{\times}$ est un isomorphisme.
\end{proof}

Il suit que $T$ est quasi-trivial. En combinant le lemme \ref{lemC16} avec \cite[Proposition~2.3]{MR3778194}, il existe un isomorphisme $T\simeq\Res_{A/k}(\mathbf{G}_m)$, où $A$ est étale sur $k$, via l'identification duquel le morphisme $\phi_0:V_0\rightarrow T$ s'étend en un morphisme $\phi_2:G\rightarrow\Res_{A/k}(\mathbf{G}_a)$. Notons $R$ la variété torique standard de $\Res_{A/k}(\mathbf{G}_a)$. On a alors:

\begin{prop}\label{extphi}
Il existe un ouvert $Y$ de $G$ stable sous l'action de $T$ et tel que $G\xrightarrow{\phi_2}\Res_{A/k}(\mathbf{G}_a)$ se restreigne à un morphisme $Y\xrightarrow{\phi_1} R$, lisse, à fibres géométriquement intègres et tel que $Y\cap\phi_2^{-1}(T)=V_0$.
\end{prop}
\begin{proof}
On applique \cite[Proposition 2.2]{MR3778194}, où $Z=\Res_{A/k}(\mathbf{G}_a)$, $X=G$, $U=V_0$ et $f=\phi_2$.
\end{proof}

Prenons maintenant $k=K$: comme $K$ est un corps $C_1$ (théorème de Tsen), tout groupe connexe est quasi-déployé (voir \cite[Corollary 25.53]{MR3729270}). On garde donc dans la suite les notations de la discussion précédente.

\begin{proof}[Démonstration du théorème \ref{thpurgrp}]
D'après la proposition \ref{extphi}, le morphisme $\phi_0:V_0\rightarrow T$ s'étend en un morphisme $\phi_1:Y\rightarrow R$. Il suffit de montrer que tout ouvert de $Y$ contenant~$Y^{(1)}$ vérifie l'approximation forte hors de $S$.

Appliquons à $\phi_1$ la méthode de fibration pour la pureté de la proposition \ref{purmethfib}, en posant $W=T$.

La condition (i) est vérifiée en vertu du corollaire \ref{AFtor}.

De plus, si $x\in W(K)$ on déduit de (\ref{defphi0}) et (\ref{defphi}) que $\phi_1^{-1}(x)$ est un espace affine sur $K$: il suit alors de la proposition \ref{affAf} que tout ouvert de $\phi_1^{-1}(x)$ contenant ${\phi_1^{-1}(x)}^{(1)}$ vérifie l'approximation forte hors de $S$, ce qui assure que la condition (ii) est vérifié.

Enfin, la condition (iii) suit immédiatement du fait que $f^{-1}(x)$ est un ouvert d'un espace affine sur $\kappa(x)$ dès que $x\in W$.
\end{proof}

\subsection{Cas général}

Rappelons l'énoncé du théorème \ref{thpureh} qu'il s'agit de démontrer dans cette sous-section:

\thpureh*

\begin{proof}[Démonstration du théorème \ref{thpureh}]
Soit $G$ un $K$-groupe semi-simple, $X$ un espace homogène sous $G$ à gauche. Considérons $U$ un ouvert de $X$ contenant $X^{(1)}$ et $\alpha\in U(\mathbf{A}^S_K)$. Quitte à remplacer~$G$ par son revêtement universel, on peut d'emblée supposer que $G$ est semi-simple et simplement connexe.

Puisque $K$ est $C_1$, on a $X(K)\neq\emptyset$ (voir \cite[III \S 2.3, Théorème 1' et III \S 2.4, Corollaire~$1$]{MR1324577}) donc $X\simeq G/H$ où $H$ est un sous-groupe de $G$. On considère alors le~$H$-torseur à droite $G\rightarrow X$ obtenu par choix d'un point rationnel de $X$. Il s'agit d'appliquer la méthode de descente \ref{purdesc} à ce torseur.

D'une part, pour $\left[\xi\right]\in \coh^1(K,H)$, le tordu $G^{\xi}$ est un $G$-torseur à gauche: il admet un point rationnel d'après \cite[\textit{ibidem}]{MR1324577} et il est trivial puisque c'est un $G$-torseur. Le théorème \ref{thpurgrp} assure donc que $G^{\xi}$ vérifie la pureté de l'approximation forte hors de $S$. D'autre part, $G^\xi(K_v)\neq\emptyset$ pour $v\in S$, puisque $G^\xi$ admet un point rationnel. Il suit alors de la proposition \ref{purdesc} que $X$ vérifie également la pureté de l'approximation forte hors de $S$, ce qui conclut.
\end{proof}

\subsection{Application à la pureté de l'approximation forte pour certaines variétés}

Nous utilisons ici les résultats précédents pour obtenir des résultats de pureté pour certaines variétés.

\subsubsection*{Les quadriques affines lisses}

Le corollaire suivant établit un résultat de pureté de l'approxi\-mation forte hors d'une place pour les quadriques affines lisses et généralise \cite[Corollaire 4.1]{MR3782219}:

\quadaffcor
\begin{proof}
Une telle quadrique affine est un espace homogène sous $SO(q)$, qui est un groupe semi-simple. On conclut donc par le théorème \ref{thpureh} qu'elle vérifie la pureté de l'approximation forte hors de $S$.
\end{proof}

Remarquons que le cas $n=3$ du corollaire précédent n'est pas recouvert par le théorème~\ref{thpuricl}: en effet, la condition sur les degrés du théorème \ref{thpuricl} permet uniquement de démontrer le corollaire précédent dans le cas $n\geq4$.

\subsubsection*{Complémentaire d'une quadrique projective lisse}

On déduit du cas des quadriques affines lisses la pureté de l'approximation forte hors d'une place pour le complémentaire d'une quadrique projective lisse, généralisant de la sorte \cite[Corollaire 4.5]{MR3782219}:

\begin{cor}\label{logquad}
Soit $S$ un ensemble fini non vide de places de $K$. Soit $n\geq3$ et $q(x_0,\dots,x_n)$ une forme quadratique non dégénérée sur $K$. Le complémentaire de l'hypersurface $q=0$ dans $\mathbf{P}^n_K$ vérifie la pureté de l'approximation forte hors de $S$.
\end{cor}
\begin{proof}
Notons $X$ le complémentaire de l'hypersurface $q=0$ dans $\mathbf{P}^n_K$ et soit $Y\subset\mathbf{A}^{n+1}_K$ la quadrique d'équation $q=1$. La projection $Y\xrightarrow{f} X$ fait de $Y$ un $\mathbf{\mu}_2$-torseur sur~$X$. Or, on a $\coh^1(K,\mathbf{\mu}_2)=K^{\times}/(K^{\times})^2$, donc le tordu de $f$ par $[c]\in \coh^1(K,\mathbf{\mu}_2)$ a son espace total défini par l'équation $q=c$. Il suit alors du corollaire \ref{quadaff} et de la proposition \ref{purdesc} que l'approximation forte vaut pour tout ouvert de $X$ contenant les points de codimension~$1$.
\end{proof}

\section{Intersections complètes affines lisses}

Cette section est dédiée à la démonstration du théorème \ref{thpuricl}, dont on rappelle l'énoncé:

\thpuricl*

\subsection{Un peu d'algèbre commutative}\label{algcomm}
Cette sous-section sert d'outil à la suivante, et peut être sautée en première lecture.

Soit $k$ un anneau commutatif. Une $k$-paire est un couple $(A,I)$ où $A$ est une $k$-algèbre et $I$ est un idéal de $A$. Un morphisme de $k$-paires $f:(A,I)\rightarrow(B,J)$ est un morphisme de $k$-algèbres $f:A\rightarrow B$ tel que $f(I)\subset J$.

Fixons dès à présent deux morphismes de $k$-paires $g:(A,I)\rightarrow (B',J')$ et $h:(B,J)\rightarrow(B',J')$. Notons $N$ le noyau de $h:B\rightarrow B'$ et supposons dorénavant que $N^2=0$. Dans la suite, on note $d:A\rightarrow\Omega^1_{A/k}$ le morphisme de différentiation canonique. On dit qu'un morphisme de $k$-paires $\phi:(A,I)\rightarrow(B,J)$ est un \textit{relèvement} de $g$ via $h$ s'il fait commuter le diagramme suivant:
\begin{center}
\begin{tikzcd}[column sep= 3em, row sep= 3 em]
      & (B,J)\arrow[d, "h"]\\
(A,I)\arrow[r, "g"] \arrow[ur, "\phi "]& (B',J')
\end{tikzcd}.
\end{center}
Notons que dans une telle configuration, le morphisme $\phi$ munit $N$ d'une structure de $A$-module et que cette structure est indépendante du choix du relèvement (car la différence de deux tels relèvements est à valeurs dans $N$, où l'idéal $N^2$ est nul). Si un relèvement de~$g$ via $h$ existe, on note $\Der_k(A,N;I\rightarrow J\cap N)$ l'ensemble des $k$-dérivations de~$A$ dans~$N$ envoyant~$I$ dans $J\cap N$. On dispose alors du fait suivant:
\begin{fait}\label{diffrelev}
S'il est non vide, l'ensemble des relèvements de $g$ via $h$ est un espace homogène principal sous $\Der_k(A,N;I\rightarrow J\cap N)$.
\end{fait}
\begin{proof}
Si $\psi$ est un autre relèvement de $g$ via $h$, l'application $\phi-\psi$ est une $k$-dérivation de $A$ dans~$N$ (voir \cite[\S25]{zbMATH00043569}) envoyant $I$ dans $J\cap N$. Réciproquement, si~$\delta$ est une $k$-dérivation de~$A$ dans $N$ telle que $\delta(I)\subset J\cap N$, l'application $\phi+\delta$ est un morphisme de $k$-paires $\phi+\delta:(A,I)\rightarrow (B,J)$ qui est également un relèvement de~$g$.
\end{proof}

Posons $\overline{N}$ l'image de $N$ par le quotient $B\rightarrow B/J$ et remarquons qu'elle est munie d'une structure de $A/I$-module si $h(J)=J'$. En faisant l'hypothèse que $h(J)=J'$, on va construire un isomorphisme naturel
\begin{equation}\label{eqn:isonat}
\Der_k(A,N;I\rightarrow J\cap N)\simeq \ker\left(\Hom_A(\Omega^1_{A/k},N)\rightarrow \Hom_{A/I}(I/I^2,\overline{N})\right)\tag{$\star$}
\end{equation}
où le morphisme $\tau:\Hom_A(\Omega^1_{A/k},N)\rightarrow \Hom_{A/I}(I/I^2,\overline{N})$ est défini comme suit: il envoie un morphisme de $A$-modules $\theta:\Omega^1_{A/k}\rightarrow N$ sur la composé de la ligne du bas dans le diagramme commutatif suivant
\begin{center}
\begin{tikzcd}
I \arrow[r,"d\vert_I"]\arrow[d]&\Omega^1_{A/k}\arrow[r, "\theta"]\arrow[d, "\pi", swap] & N\arrow[d, "q"]\\
I/I^2\arrow[r,"\alpha"] & \Omega^1_{A/k}/I\Omega^1_{A/k}\arrow[r,"\overline{\theta}"] & \overline{N}
\end{tikzcd}
\end{center}
où les flèches verticales sont les applications quotients, le morphisme $\alpha$ est défini par $\alpha([i])=[di]$ et le morphisme $\overline{\theta}$ est obtenu par passage au quotient de $q\circ\theta$ du fait que $\theta(I\Omega^1_{A/k})\subset IN\subset JN$ et donc $q\circ\theta(I\Omega^1_{A/k})\subset q(JN)=0$.

Pour construire l'isomorphisme \eqref{eqn:isonat}, commençons par remarquer qu'on a un isomorphisme naturel $$\Der_k(A,N;I\rightarrow J\cap N)\simeq \Hom_A(\Omega^1_{A/k},N;dI\rightarrow J\cap N)$$où $\Hom_A(\Omega^1_{A/k},N;dI\rightarrow J\cap N)$ désigne le sous-module de $\Hom_A(\Omega^1_{A/k},N)$ constitué des morphismes envoyant $dI$ dans $J\cap N$. Or, par construction de $\tau$, le noyau $\ker(\tau)$ correspond précisément aux morphismes $\theta:\Omega^1_{A/k}\rightarrow N$ tels que $\theta(dI)\subset J\cap N$, d'où l'isomorphisme naturel souhaité.
\begin{prop}\label{propdeform}
Avec les notations précédentes et en supposant que $h(J)=J'$, il existe un isomorphisme naturel $$\Der_k(A,N;I\rightarrow J\cap N)\simeq \ker\left(\Hom_A(\Omega^1_{A/k},N)\rightarrow \Hom_{A/I}(I/I^2,\overline{N})\right).$$
\end{prop}

Pour l'existence des relèvements, on dispose du théorème suivant:

\begin{thm}\label{thmdeform}
Supposons que $A$ et $A/I$ sont formellement lisses sur $k$ (au sens de de \cite[Définition (17.1.1)]{zbMATH03245973}), que $h$ est surjective avec $h(J)=J'$. Alors il existe un relèvement de $g$ via $h$.
\end{thm}
\begin{proof}
Considérons le diagramme commutatif de $k$-algèbres suivant
\begin{center}
\begin{tikzcd}
& B \arrow[r, twoheadrightarrow, "p"]\arrow[d, twoheadrightarrow, "h"] & B/J \arrow[d,"\overline{h}", twoheadrightarrow] \\
A\arrow [r, "g"] & B' \arrow[r, twoheadrightarrow, "p'"] & B'/J'
\end{tikzcd}
\end{center}
où $p$ et $p'$ sont les applications quotient et où $\overline{h}$ est obtenue par passage au quotient de~$p'\circ h$. Puisque par hypothèse le noyau $N$ de $h$ est de carré nul et puisque $A$ est formellement lisse sur $k$, il existe un morphisme de $k$-algèbre $\phi: A\rightarrow B$ s'insérant dans le diagramme commutatif précédent. Notons d'emblée que puisque $g(I)\subset J'$, on a $\phi(I)\subset J+N$. 

Pour construire un relèvement de $k$-paires de $g$ via $h$, il suffit d'après \cite[\S25, p.191]{zbMATH00043569} de trouver une dérivation $\delta\in \Der_k(A,N)$ telle que $(\phi+\delta)(I)\subset J$, auquel cas $\phi+\delta$ définit un relèvement de $k$-paires de $g$ via $h$. Or, le morphisme $p\vert_{N,*}:\Der_k(A,N)\rightarrow \Der_k(A,\overline{N})$ étant naturellement isomorphe au morphisme $p\vert_{N,*}:\Hom_A(\Omega^1_{A/k},N)\rightarrow \Hom_A(\Omega^1_{A/k},\overline{N})$, il est surjectif car $p\vert_N:N\rightarrow\overline{N}$ est surjective et $\Omega^1_{A/k}$ est un $A$-module projectif par \cite[Proposition~17.2.3.(i)]{zbMATH03245973} combiné à \cite[\href{https://stacks.math.columbia.edu/tag/05JQ}{Tag 05JQ}]{stacks-project}. Il suffit donc de trouver $\epsilon\in \Der_k(A,\overline{N})$ telle que $(p\circ\phi-\epsilon)(I)=0$ car dans ce cas tout antécédent $\delta$ de $\epsilon$ par $p\vert_{N,*}$ convient.

En posant $\psi=p\circ\phi$, on cherche donc $\epsilon\in \Der_k(A,\overline{N})$ telle que $\epsilon\vert_I=\psi\vert_I$. Par l'identification naturelle $\Der_k(A,\overline{N})\simeq \Hom_A(\Omega^1_{A/k},\overline{N})$, ceci revient à chercher un morphisme $\gamma:\Omega^1_{A/k}\rightarrow\overline{N}$ tel que $\gamma\circ d\vert_I=\psi\vert_I$. Considérons maintenant le diagramme commutatif suivant
\begin{center}
\begin{tikzcd}
\overline{N}&I \arrow[l,"\psi\vert_I"] \arrow[r, "d"]\arrow[d] & \Omega^1_{A/k}\arrow[d,"\pi"]\\
&I/I^2\arrow[r,"\alpha"]\arrow[ul, "\beta"] & \Omega^1_{A/k}/I\Omega^1_{A/k}
\end{tikzcd}
\end{center}
où $\beta$ est obtenue par passage au quotient de $\psi\vert_I$ modulo $I^2$ (car $\psi(I^2)=p(N^2)=0$) et où les flèches du carré de droite ont été définies dans la discussion précédant la proposition~\ref{propdeform}. Or, $A/I$ étant formellement lisse sur $k$, le morphisme $\alpha$ admet une rétraction $\rho$ d'après \cite[Theorem 25.2 (5)]{zbMATH00043569}. Il suit alors que le morphisme $\gamma=\beta\circ\rho\circ\pi$ convient, ce qui conclut.
\end{proof}

\subsection{Préambule géométrique}\label{espacetgt}

Exception faite du lemme \ref{lemmor}, tout au long de cette section $Y$ (resp. $X$) désigne une variété projective (resp. quasi-projective) sur un corps algébriquement clos $k$. On considère $E$ (resp. $D$) une sous-variété fermée de $Y$ (resp. de~$X$). On suppose que $X$ et $D$ sont lisses. Notons indifféremment $M$ ou $\Mor(Y,X;im(E)\subset D)$ le sous-schéma fermé $$\Mor(Y,X)\times_{\Mor(E,X)}\Mor(E,D)$$ de $\Mor(Y,X)$ paramétrant les morphismes $f:Y\rightarrow X$ tels que $E\subset f^{-1}(D)$.

Fixons désormais un tel morphisme $f$. On se propose d'étudier le schéma $M$ localement en $[f]$. Commençons par décrire l'espace tangent de $M$ en $\left[f\right]$. Partant du fait que $M=\Mor(Y,X)\times_{\Mor(E,X)}\Mor(E,D)$, on a $$T_{[f]}M=T_{[f]}\Mor(Y,X)\times_{T_{\left[f\vert_E\right]}\Mor(E,X)}T_{\left[f\vert_E\right]}\Mor(E,D)$$ où les facteurs du produit fibré se réécrivent sous la forme:
\begin{itemize}
\item $T_{\left[f\right]}\Mor(Y,X)=\coh^0(Y,f^*\mathscr{T}_X)$;
\item $T_{\left[f\vert_E\right]}\Mor(E,X)=\coh^0(E,(f\vert_E)^*\mathscr{T}_X)=\coh^0(E,\left(f^*\mathscr{T}_X\right)\vert_E)$;
\item $T_{\left[f\vert_E\right]}\Mor(E,D)=\coh^0(E,(f\vert_E)^*\mathscr{T}_D)$;
\end{itemize}
ces dernières égalités sur les espaces tangents reposant sur \cite[\S$2.3$]{MR1841091}.
Puisque $D$ est une sous-variété lisse de $X$, on a la suite exacte de faisceaux localement libres sur $D$:
\begin{center}
\begin{tikzcd}
0\arrow[r] & \mathscr{T}_D \arrow[r] & \mathscr{T}_X\vert_D \arrow[r] & \mathscr{N}_{D/X} \arrow[r] & 0
\end{tikzcd}
\end{center}
puis en tirant en arrière par $f\vert_E$ on obtient la suite exacte de faisceaux localement libres sur $E$:
\begin{center}
\begin{tikzcd}
0\arrow[r] & (f\vert_E)^*\mathscr{T}_D \arrow[r] & (f\vert_E)^*(\mathscr{T}_X\vert_D) \arrow[r] & (f\vert_E)^*\mathscr{N}_{D/X} \arrow[r] & 0
\end{tikzcd}.
\end{center}
Mais puisque $(f\vert_E)^*\left(\mathscr{T}_X\vert_D\right)=(f^*\mathscr{T}_X)\vert_E$, cette dernière induit une suite exacte:
\begin{center}
\begin{tikzcd}
0\arrow[r] & \coh^0(E,(f\vert_E)^*\mathscr{T}_D) \arrow[r] & \coh^0(E,\left(f^*\mathscr{T}_X\right)\vert_E) \arrow[r] & \coh^0(E,(f\vert_E)^*\mathscr{N}_{D/X})
\end{tikzcd}.
\end{center}
Enfin, comme le morphisme $T_{\left[f\right]}\Mor(Y,X)\rightarrow T_{\left[f\vert_E\right]}\Mor(E,X)$ est le morphisme de restriction $\coh^0(Y,f^*\mathscr{T}_X)\rightarrow \coh^0(E,\left(f^*\mathscr{T}_X\right)\vert_E)$, on déduit de la suite exacte précédente: $$T_{\left[f\right]}M=\coh^0(Y,\ker(f^*\mathscr{T}_X\rightarrow \iota_{E,*}(f\vert_E)^*\mathscr{N}_{D/X})).$$

Posons $\mathscr{F}=\ker(f^*\mathscr{T}_X\rightarrow \iota_{E,*}(f\vert_E)^*\mathscr{N}_{D/X})$. Nous prouvons maintenant le résultat principal de cette sous-section, qui fournit notamment un critère de lisseté de $M$ en $[f]$. Il s'agit d'une généralisation de \cite[Theorem 2.6]{MR1841091}.

\begin{thm}\label{thmdeform}
Soit $k$ un corps algébriquement clos et $Y$ (resp. $X$) une variété projective (resp. quasi-projective) sur $k$. Considérons également $E$ (resp. $D$) une sous-variété fermée de $Y$ (resp. de $X$). On suppose que $X$ et $D$ sont lisses. Localement en~$[f]$, le schéma $M=\Mor(Y,X;im(E)\subset D)$ peut alors être défini par $h^1(Y,\mathscr{F})$ équations au sein d'une variété lisse de dimension $h^0(Y,\mathscr{F})$. En particulier, toute composante irréductible de $M$ passant par $[f]$ est de dimension au moins $$h^0(Y,\mathscr{F})-h^1(Y,\mathscr{F}).$$
\end{thm}
La démonstration que nous donnons du théorème précédent suit formellement celle donnée par Debarre de \cite[\textit{ibidem}]{MR1841091} puis nécessite le lemme crucial \ref{lemmor} dont la démonstration justifie le développement fait au sein de la sous-section \ref{algcomm}.
\begin{proof}
Considérons un voisinage affine $U$ de $[f]$ dans $M$, défini par des équations polynomiales $P_1,\dots,P_m$ dans un espace affine $\mathbf{A}^n_k$. Notons $r$ le rang de la matrice jacobienne $(\partial P_i/\partial x_j([f]))$ et, quitte à réordonner les $P_i$, supposons que les $r$ premières lignes sont indépendantes. Posons alors $V$ la sous-variété de $\mathbf{A}^n_k$ définie par $P_1,\dots,P_r$: elle est lisse en $[f]$, et puisque $r$ est à la fois la codimension de $T_{[f]}M$ et de $T_{[f]}V$ dans $k^n$, il vient que~$T_{[f]}M=T_{[f]}V$.

En posant $h^i=h^i(Y,\mathscr{F})$, montrons que dans un voisinage de $[f]$, la variété $M$ peut être définie par $h^1$ équations dans la variété $V$ de dimension $h^0$. Pour ce faire, il suffit de montrer que l'idéal $I$ de l'anneau local $R=\mathscr{O}_{V,[f]}$ définissant le fermé $\spec(\mathscr{O}_{M,[f]})=U\times_V\spec(R)$ de $\spec(R)$ est engendré par $h^1$ éléments: en utilisant le lemme de Nakayama, il suffit alors de montrer que~$I/\mathfrak{m}I$ est un $k$-espace vectoriel de dimension au plus $h^1$. Notons également que puisque $V$ et $M$ ont même espace tangent en $[f]$, l'idéal $I$ est inclus dans le carré de l'idéal maximal~$\mathfrak{m}$ de $\mathscr{O}_{V,[f]}$.

Considérons le morphisme canonique $\spec(\mathscr{O}_{M,[f]})=\spec(R/I)\rightarrow M$. Il correspond à un $k$-morphisme $f_{R/I}:Y\otimes_kR/I\rightarrow X$ étendant $f$ et tel que $E\otimes_kR/I\subset f_{R/I}^{-1}(D)$. D'après le lemme \ref{lemmor} ci-dessous, l'obstruction à étendre $f_{R/I}$ à un morphisme $f_{R/\mathfrak{m}I}:Y\otimes_k R/\mathfrak{m}I\rightarrow X$ tel que $E\otimes_k R/\mathfrak{m}I\subset f_{R/\mathfrak{m}I}^{-1}(D)$ vit dans $\coh^1(Y,\mathscr{F})\otimes_k(I/\mathfrak{m}I)$. Écrivons donc cette obstruction sous la forme $$\sum_{i=1}^{h^1}a_i\otimes \overline{b_i}$$
où $(a_1,\dots,a_{h^1})$ est une $k$-base de $\coh^1(Y,\mathscr{F})$ et les $b_i$ sont dans $I$. Cet élément est alors nul dans 
$$\coh^1(Y,\mathscr{F})\otimes_k(I/(\mathfrak{m}I+(b_1,\dots,b_{h^1})))$$
ce qui signifie que $f_{R/I}$ s'étend en un morphisme $Y\otimes_k R/(\mathfrak{m}I+(b_1,\dots,b_{h^1}))\rightarrow X$ par lequel l'image réciproque de $D$ contient $E\otimes_kR/(\mathfrak{m}I+(b_1,\dots,b_{h^1}))$. Il s'ensuit que le morphisme canonique $\spec(R/I)\rightarrow M$ se factorise par un morphisme $\spec(I/(\mathfrak{m}I+(b_1,\dots,b_{h^1})))\rightarrow M$. Autrement dit, le morphisme identité de $R/I$ se factorise sous la forme
\begin{center}
\begin{tikzcd}
R/I \arrow[r] & R/(\mathfrak{m}I+(b_1,\dots,b_{h^1})) \arrow[r,"\pi"] & R/I
\end{tikzcd}
\end{center}
où $\pi$ est la projection canonique. Le fait \ref{faitalgcomm} assure alors que $I=(\mathfrak{m}I+(b_1,\dots,b_{\coh^1}))$. Ceci implique que $I/\mathfrak{m}I$ est engendré par la famille $(b_1,\dots,b_{h^1})$, ce qui conclut.
\end{proof}

\begin{lem}\label{lemmor}
Soit $Y$ (resp. $X$) une variété sur corps algébriquement clos $k$. Considérons~$E$ (resp. $D$) une sous-variété fermée de $Y$ (resp. $X$). Supposons $X$ et $D$ lisses et choisissons $f:Y\rightarrow X$ telle que $E\subset f^{-1}(D)$. Soit $R$ une $k$-algèbre locale de type fini, d'idéal maximal~$\mathfrak{m}$ et de corps résiduel $k$. Soit $I$ un idéal de $R$ tel que $\mathfrak{m}I=0$. Considérons $f_{R/I}:Y\otimes_kR/I\rightarrow X$ un $k$-morphisme étendant $f$ et tel que $E\otimes_k R/I\subset f_{R/I}^{-1}(D)$.
\begin{enumerate}[label=(\arabic*)]
\item Si $X$ et $Y$ sont affines, alors $f_{R/I}$ s'étend à un $k$-morphisme $f_R:Y\otimes_kR\rightarrow X$ tel que $E\otimes_kR\subset f_R^{-1}(D)$ et deux telles extensions diffèrent d'un élément de $\coh^0(Y,\mathscr{F})\otimes_kI$.
\item L'obstruction à étendre $f_{R/I}$ à un $k$-morphisme $f_R:Y\otimes_kR\rightarrow X$ tel que $E\otimes_kR\subset f_R^{-1}(D)$ vit dans $\coh^0(Y,\mathscr{F})\otimes_kI$.
\end{enumerate}
\end{lem}
\begin{proof}
Montrons $(1)$ et supposons donc que $Y$ et $X$ sont affines avec $Y=\spec(B)$, $X=\spec(A)$, $E$ défini par un idéal $J_E$ de $B$ et $D$ par un idéal $J_D$ de $A$. L'hypothèse sur~$f_{R/I}$ assure que le morphisme de $k$-algèbres $f^{\#}_{R/I}:A\rightarrow B\otimes_kR/I$ associé à $f_{R/I}$ définit un morphisme de $k$-paires $$f^{\#}_{R/I}:(A,J_D)\rightarrow(B\otimes_kR/I,J_E\otimes_kR/I).$$On cherche alors un morphisme de $k$-paires $f^{\#}_R:(A,J_D)\rightarrow(B\otimes_kR,J_E\otimes_kR)$ s'inscrivant dans le diagramme commutatif suivant
\begin{equation}\label{formuleh0}
\begin{tikzcd}[column sep= 3em, row sep= 3 em]
& (B\otimes_kR,J_E\otimes_kR) \arrow[d, "h"] \\
(A,J_D) \arrow[r,"f_{R/I}^{\#}",swap]\arrow[ur,"f^{\#}_R",dashrightarrow] & (B\otimes_kR/I,J_E\otimes_kR/I)
\end{tikzcd}\tag{$\star\star$}
\end{equation}
où la flèche verticale désigne le morphisme quotient, de noyau $B\otimes_kI$. Puisque les algèbres~$A$ et $A/J_D$ sont $k$-lisses, elles sont formellement lisses sur $k$, et l'hypothèse $\mathfrak{m}I=0$ implique que le noyau $N=B\otimes_kI$ de $h$ est de carré nul: le théorème \ref{thmdeform} assure alors l'existence d'un relèvement $f^{\#}_R$. En outre, puisque l'image de $B\otimes_kI$ par la projection dans $B/J_E\otimes_kR$ est $B/J_E\otimes_k I$, le fait \ref{diffrelev} combiné à la proposition \ref{propdeform} assure que deux tels relèvements de $k$-paires diffèrent d'un élément de
$$\ker\left(\Hom_A(\Omega^1_{A/k},B\otimes_kI)\rightarrow \Hom_{A/J_D}(J_D/J_D^2,B/J_E\otimes_k I)\right)$$
où l'hypothèse $\mathfrak{m}I=0$ assure que la structure de $A$-module sur $B\otimes_kI$ est celle donnée par $a.(b\otimes i)=(ab)\otimes i$. De même, la structure de $A/J_D$-module de $B/J_E\otimes_kI$ est héritée de celle de $B/J_E$.  Par platitude de $I$ sur $k$, deux tels relèvements de $k$-paires de $f^{\#}_{R/I}$ via~$h$ diffèrent donc d'un élément de
$$\left[\ker\left(\Hom_A(\Omega^1_{A/k},B)\rightarrow \Hom_{A/J_D}(J_D/J_D^2,B/J_E)\right)\right]\otimes_kI.$$

Montrons que $\ker\left(\Hom_A(\Omega^1_{A/k},B)\rightarrow \Hom_{A/J_D}(J_D/J_D^2,B/J_E)\right)=\coh^0(Y,\mathscr{F})$. En effet, d'une part $\coh^0(X,\mathscr{T}_X)=\Hom_A(\Omega^1_{A/k},A)$ et donc $$\coh^0(Y,f^*\mathscr{T}_X)=\Hom_A(\Omega^1_{A/k},A)\otimes_AB=\Hom_A(\Omega^1_{A/k},B)$$ où la dernière égalité provient du fait que $\Omega^1_{A/k}$ est projectif, par lisseté formelle de $A/k$. D'autre part, puisque $\coh^0(D,\mathscr{N}_{D/X})=\Hom_{A/J_D}(J_D/J_D^2,A/J_D)$, il vient que $$\coh^0(E,f\vert_E^*\mathscr{N}_{D/X})=\Hom_{A/J_D}(J_D/J_D^2,A/J_D)\otimes_{A/J_D}B/J_E=\Hom_{A/J_D}(J_D/J_D^2,B/J_E)$$où la dernière égalité suit du fait que le $A/J_D$-module conormal $J_D/J_D^2$ est projectif (voir \cite[Chapter II, Theorem 8.17.(2)]{MR0463157}). On en déduit que
$$\coh^0(Y,\iota_{E,*}f\vert_E^*\mathscr{N}_{D/X})=\Hom_{A/J_D}(J_D/J_D^2,B/J_E)$$
vu comme $B$-module, et donc que
$$\coh^0(Y,\mathscr{F})=\ker\left(\Hom_A(\Omega^1_{A/k},B)\rightarrow \Hom_{A/J_D}(J_D/J_D^2,B/J_E)\right).$$
Ainsi, deux relèvements de $k$-paires de $f^{\#}_{R/I}$ via $h$ dans le diagramme (\ref{formuleh0}) diffèrent d'un élément de~$\coh^0(Y,\mathscr{F})\otimes_kI$.\\

Pour montrer $(2)$, soient $X=\bigcup_iU_i$ et $Y=\bigcup_iV_i$ deux recouvrements affines finis de~$X$ et $Y$ tels que $f(V_i)\subset U_i$. Par $(1)$, il existe pour tout $i$ une extension $f_{R,i}:V_i\otimes_kR\rightarrow U_i$ de $f_{R/I}\vert_{V_i\otimes_kR/I}$. On sait également par $(1)$ que sur $V_i\cap V_j$, les relèvements $f_{R,i}$ et $f_{R,j}$ diffèrent d'un élément de $\coh^0(V_i\cap V_j,\mathscr{F})\otimes_kI$. Par construction de la cohomologie de \v{C}ech, la famille de ces éléments correspond à un $1$-cocycle et donc à un élément de $\coh^1(Y,\mathscr{F})\otimes_kI$ ce qui conclut.
\end{proof}

\begin{fait}\label{faitalgcomm}\cite[\S2.2, Lemma 2.8]{MR1841091}
Soit $R$ un anneau local noethérien d'idéal maximal $\mathfrak{m}$ et soit $I$ un idéal de $R$ contenu dans $\mathfrak{m}^2$. Si le morphisme quotient $R\rightarrow R/I$ admet une section, alors $I=0$.
\end{fait}

\subsection{Un lemme de déformation}\label{lemdeform}
Dans cette sous-section, fixons $X$ une variété quasi-projective lisse sur $K$ plongée dans $\mathbf{P}^n_K$, $D$ un diviseur lisse de $X$ et $E$ un sous-schéma fermé strict de $\mathbf{P}^1_K$ dont l'inclusion $E\xhookrightarrow{}\mathbf{P}^1_K$ est notée $\iota_E$. Notons $d$ la dimension de $X$.

Considérons le foncteur envoyant un $K$-schéma $T$ sur l'ensemble des $K$-morphismes $f:\mathbf{P}^1_K\times_KT\rightarrow X$ tels que:
\begin{enumerate}[label=(\roman*)]
\item le morphisme $f$ est de degré $2$, c'est-à-dire que $f^*(\mathscr{O}_X(1))$ est un fibré en droites isomorphe à $\mathscr{O}_{\mathbf{P}^1_T}(2)$;
\item $f(\{0\}\times T)$ et $f(\{1\}\times T)$ sont inclus dans $X\backslash D$;
\item $E\times_KT\subset f^{-1}(D)$.
\end{enumerate}

Remarquons que ce foncteur est représentable. En effet, la condition (i) définit le sous-schéma ouvert $\Mor_2(\mathbf{P}^1_K,X)$ de $\Mor(\mathbf{P}^1_K,X)$ paramétrant les morphismes de $\mathbf{P}^1_K$ dans~$X$ de degré~$2$ (on peut se référer à \cite[\S2.1]{MR1841091} pour le fait que c'est un sous-schéma ouvert). La condition (ii) définit le sous-schéma ouvert de $\Mor(\mathbf{P}^1_K,X)$ paramétrant les morphismes de $\mathbf{P}^1_K$ dans $X$ envoyant $0$ et $1$ hors de $D$. Enfin, la condition (iii) définit le fermé $M=\Mor(\mathbf{P}^1_K,X;im(E)\subset D)$ de~$\Mor(\mathbf{P}^1_K,X)$ étudié dans la sous-section précédente. Ce foncteur est donc représenté par un sous-schéma ouvert $H$ de $M$.

On dispose alors d'un morphisme d'évaluation
$$ev_{0,1}:H\rightarrow X\times X$$
envoyant un morphisme $f$ sur $(f(0),f(1))$, ainsi que d'un morphisme d'évaluation$$ev:\mathbf{P}^1_K\times H\rightarrow X$$envoyant $(p,\left[f\right])$ sur $f(p)$.\\

Cette sous-section est dédiée à la démonstration du lemme géomé\-trique \ref{lemkollar} portant sur la différentielle de $ev$ sur $(\mathbf{P}^1_K\backslash E)\times H$, et dont la démonstration est une variante de celle de \cite[Chapter II.3, Proposition~3.10]{MR1440180}. Ce lemme est utilisé plus bas pour démontrer le théorème \ref{purafmod}. Avant de l'énoncer, rappelons que $H$ étant un ouvert de$$M=\Mor(\mathbf{P}^1_K,X;im(E)\subset D)$$il vient que pour tout point géométrique $[f]\in H$ on a $T_{[f]}H=T_{[f]}M$ ce qui, d'après la sous-section \ref{espacetgt}, donne l'identification:
$$T_{[f]}H=\coh^0(\mathbf{P}^1_K,\mathscr{F})$$
où $\mathscr{F}=\ker(f^*\mathscr{T}_X\rightarrow \iota_{E,*}(f\vert_E)^*\mathscr{N}_{D/X})$ est localement libre, car sous-faisceau du faisceau localement libre $f^*\mathscr{T}_X$ sur la courbe lisse $\mathbf{P}^1_K$. Remarquons également que $\mathscr{F}$ coïncide avec~$f^*\mathscr{T}_X$ sur $\mathbf{P}^1_K\backslash E$, si bien que $\mathscr{F}$ est localement libre de rang $d=\dim(X)$. Par \cite[Theorem $4.1$]{MR662762} il existe des entiers $a_1,\dots,a_d$ et un isomorphisme $$\mathscr{F}\simeq\mathscr{O}(a_1)\oplus\dots\oplus\mathscr{O}(a_d).$$ En outre, pour $q\in\mathbf{P}^1_K\backslash E$, notons $res_q:\coh^0(\mathbf{P}^1_K,\mathscr{F})\rightarrow\mathscr{F}=T_{f(q)}X$ le morphisme de restriction en $q$, qui s'identifie au morphisme
\begin{center}
\begin{tikzcd}[row sep=0.1pt]
\bigoplus_{1\leq i\leq d}\coh^0(\mathbf{P}^1_K,\mathscr{O}(a_i))\arrow[r] & \bigoplus_{1\leq i \leq d} \mathscr{O}(a_i)\\
(s_1,\dots,s_d)\arrow[r,mapsto] & (s_1(q),\dots ,s_d(q))
\end{tikzcd}
\end{center}
et dont le rang est $\Card\{i:a_i\geq0\}$, laquelle quantité dépend de $f$ (par définition de $\mathscr{F}$) mais est indépendante de $q$.
\begin{lem}\label{lemkollar}
Supposons $E$ de longueur au plus $2$, fixons $[f]$ un point géométrique de $H$ et reprenons les notations précédant le lemme. Alors le 	rang de la différentielle de $ev$ est constant sur $(\mathbf{P}^1_K\backslash E)\times[f]$, égal à $$\Card\{i:a_i\geq0\}.$$
\end{lem}
\begin{proof}
L'énoncé est géométrique: pour éviter d'alourdir les notations, on suppose donc ici que $K$ est algébriquement clos. Choisissons $(q,[f])\in(\mathbf{P}^1(K)\backslash E(K))\times H(K)$ et reprenons les notations de la discussion précédant le lemme.

Commençons par donner une expression de la différentielle $$d_{(q,\left[f\right])}ev:T_q\mathbf{P}^1_K\oplus\coh^0(\mathbf{P}^1_K,\mathscr{F})\rightarrow T_{f(q)}X.$$
Le schéma $H$ étant localement fermé dans $\Mor(\mathbf{P}^1_K,X)$ le morphisme $d_{(q,\left[f\right])}ev$ est la restriction de la différentielle en $(q,\left[f\right])$ du morphisme d'évaluation $\mathbf{P}^1_K\times\Mor(\mathbf{P}^1_K,X)\rightarrow X$ envoyant $(x,[g])$ sur $g(x)$. Or l'expression de cette dernière étant donnée par \cite[Chapter II, Proposition~$3.4$]{MR1440180}, on obtient l'expression suivante de $d_{(q,\left[f\right])}ev$:
\begin{center}
\begin{tikzcd}[row sep=0.1pt]
d_{(q,\left[f\right])}ev: & [-3em]T_q\mathbf{P}^1_K \oplus \coh^0(\mathbf{P}^1_K,\mathscr{F})\arrow[r] & T_{f(q)}X \\
 & (t,s)\arrow[r, mapsto] & d_qf(t) + res_q(s)
\end{tikzcd}.
\end{center}

Montrons maintenant que $im(d_{(q,\left[f\right])}ev)=im(res_q)$. Pour ce faire, commençons par remarquer que $\coh^0(\mathbf{P}^1_K,f^*\mathscr{T}_X\otimes\mathscr{I}_E)$ est inclus dans $\coh^0(\mathbf{P}^1_K,\mathscr{F})$, car toute section globale de $f^*\mathscr{T}_X$ qui s'annule le long de $E$ induit une section globale nulle de $\iota_{E,*}(f\vert_E)^*\mathscr{N}_{D/X}$. De plus, comme $q$ est hors de $E$ on a le diagramme commutatif suivant
\begin{center}
\begin{tikzcd}[column sep= 5em]
\coh^0(\mathbf{P}^1_K,\mathscr{T}_{\mathbf{P}^1_K}\otimes\mathscr{I}_E) \arrow[r]\arrow[d, "df\otimes \mathscr{I}_E"] & \mathscr{T}_{\mathbf{P}^1_K}\vert_q=T_q\mathbf{P}^1_K\arrow[d,"d_qf"]\\
\coh^0(\mathbf{P}^1_K,f^*\mathscr{T}_X\otimes\mathscr{I}_E) \arrow[r]\arrow[d,hook] & f^*T_X\vert_q=T_{f(q)}X\\
\coh^0(\mathbf{P}^1_K,\mathscr{F})\arrow[ur,"res_q", swap, bend right=20] &
\end{tikzcd}
\end{center}
où les flèches horizontales sont les morphismes de restriction en $q$. Il suffit alors de montrer que $\alpha$ est surjective, car le diagramme précédent assure alors que $im(d_qf)\subset im(res_q)$, ce qui conclut. Pour montrer la surjectivité de $\alpha$, partons de la suite exacte courte
\begin{center}
\begin{tikzcd}
0 \arrow[r] & \mathscr{O}_{\mathbf{P}^1_K}(-q) \arrow[r] & \mathscr{O}_{\mathbf{P}^1_K} \arrow[r] & \mathscr{O}_q \arrow[r] & 0
\end{tikzcd}.
\end{center}
Puisque $q$ est hors de $E$ et comme le faisceau $\mathscr{T}_{\mathbf{P}^1_K}\otimes\mathscr{I}_E$ est inversible, la tensorisation par~$\mathscr{T}_{\mathbf{P}^1_K}\otimes\mathscr{I}_E$ de la suite exacte précédente donne une suite exacte
\begin{center}
\begin{tikzcd}
0 \arrow[r] & \mathscr{T}_{\mathbf{P}^1_K}\otimes\mathscr{I}_E(-q) \arrow[r] & \mathscr{T}_{\mathbf{P}^1_K}\otimes\mathscr{I}_E \arrow[r] & \mathscr{T}_{\mathbf{P}^1_K}\vert_q \arrow[r] & 0
\end{tikzcd}.
\end{center}
La surjectivité de $\alpha$ suit alors de l'annulation de $\coh^1(\mathbf{P}^1_K,\mathscr{T}_{\mathbf{P}^1_K}\otimes\mathscr{I}_E(-q))$, laquelle est assurée par le fait que $E$ est un fermé de longueur au plus $2$ de $\mathbf{P}^1_K$.

De là, pour tout $q\in\mathbf{P}^1(K)\backslash E(K)$ le rang de $d_{(q,\left[f\right])}ev$ est également celui de $res_q$, qui est égal à $\Card\{i:a_i\geq0\}$, ce qui conclut.
\end{proof}
Dans la suite, le lemme \ref{lemkollar} est également utilisé sous la forme suivante:
\begin{cor}\label{corcrucial}
Supposons $E$ de longueur au plus $2$. S'il existe un point géométrique~$(p,[f])$ de $(\mathbf{P}^1_K\backslash E)\times H$ en lequel la différentielle de $ev$ est surjective, alors $H$ est lisse en $[f]$.
\end{cor}
\begin{proof}
La différentielle de $ev$ étant surjective en $(p,[f])$, on a $$d=\Card\{i:a_i\geq0\}$$c'est-à-dire que $\mathscr{F}=\mathscr{O}(a_1)\oplus\dots\oplus\mathscr{O}(a_d)$ où les entiers $a_i$ sont positifs. Il s'ensuit que~$\coh^1(\mathbf{P}^1_K,\mathscr{F})=0$ ce qui, par le théorème \ref{thmdeform}, assure que $H$ est lisse en $[f]$.
\end{proof}

\subsection{Un critère de pureté}
Au cours de cette sous-section, on reprend les notations de la sous-section \ref{lemdeform} en supposant désormais que $D$ est une section hyperplane lisse de $X$ et que $$E=\spec(\mathscr{O}_{\mathbf{P}^1_K,\infty}/\mathfrak{m}_{\mathbf{P}^1_K,\infty}^2).$$ Ce choix de $E$ et $D$ permet d'assurer que la condition (iii) définissant $H$ dans la sous-section~\ref{lemdeform} peut être lue comme une égalité d'ensembles:
\begin{schol}\label{scholie}
Soit $f:\mathbf{P}^1_K\rightarrow X$ un morphisme dont la composée avec $X\xhookrightarrow{}\mathbf{P}^n_K$ est de degré $2$. Alors les assertions suivantes sont équivalentes:
\begin{enumerate}[label=(\roman*)]
\item $E\subset f^{-1}(D)\subsetneq \mathbf{P}^1_K$;
\item $E=f^{-1}(D)$;
\item on a l'égalité d'ensembles $\{\infty\}=f^{-1}(D)$.
\end{enumerate}
\end{schol}
\begin{proof}
Montrons que (i) implique (ii). Puisque $f(\mathbf{P}^1_K)\not\subset D$, $f^{-1}(D)$ est un diviseur effectif de $\mathbf{P}^1_K$. Ce dernier est de longueur $2$ car $f$ est de degré $2$ et $D$ une section hyperplane de~$X$. Mais $E\subset f^{-1}(D)$ et $E$ est l'unique sous-schéma fermé de longueur $2$ de~$\mathbf{P}^1_K$ contenant~$\infty$. Il en découle que $E=f^{-1}(D)$, d'où l'assertion (ii).

L'assertion (iii) suit naturellement de (ii). Supposons donc (iii) et montrons (i). L'hypothèse de (iii) assure que $f(\mathbf{P}^1_K)\not\subset D$, si bien que $f^{-1}(D)$ est un diviseur effectif de~$\mathbf{P}^1_K$. Comme précédemment, ce diviseur est de longueur $2$ par hypothèse sur $f$ et $D$. Mais comme il contient $\infty$, il est égal à $E$ ce qui conclut.
\end{proof}
Cette partie est dédiée à la démonstration du théorème suivant, qui est un critère de pureté pour l'approximation forte faisant intervenir le schéma $H$ défini au début de la sous-section \ref{lemdeform}:

\begin{thm}\label{purafmod}
Supposons que:
\begin{enumerate}[label=(\roman*)]
\item la variété $X$ vérifie l'approximation faible;
\item il existe une composante irréductible $F$ de $H$ telle que $ev_{0,1}\vert_F$ est dominante, de fibre générique géométriquement intègre et rationnellement connexe.
\end{enumerate}
Alors la pureté de l'approximation forte vaut pour $X\backslash D$ hors de tout ensemble fini non vide de places de $K$.
\end{thm}
Avant d'en donner une démonstration, nous avons besoin de quelques résultats intermédiaires. Commençons par énoncer une variante de \cite[Theorem 1.6]{zbMATH06921639} et donnons-en une démonstration proche de celle de \cite[Lemma 1.8]{MR3432584}:

\begin{prop}\label{afmod}
Faisons les mêmes hypothèses que dans le théorème \ref{purafmod}, fixons une composante irréductible $F$ comme dans (ii) et considérons un fermé $Z$ de $X$ tel que:
\begin{enumerate}[label=(\roman*)]
\setcounter{enumi}{2}
\item il existe un ouvert non vide $W$ de $F$ tel que $ev(\mathbf{A}^1_K\times W)\subset X\backslash Z$.
\end{enumerate}
Alors l'approximation forte vaut pour $X\backslash(D\cup Z)$ hors de tout ensemble fini non vide de places de $K$.
\end{prop} 

\begin{proof}
Considérons $v_0$ une place de $K$ et montrons que $U=X\backslash (D\cup Z)$ vérifie l'approximation forte hors de $v_0$. Pour ce faire, choisissons $\xi\in U(\mathbf{A}_K)$, $S$ un ensemble fini de places de $K$ contenant $v_0$ ainsi que des $\mathscr{O}_{K,S}$-modèles $\mathscr{X}$, $\mathscr{D}$ et $\mathscr{Z}$ respectifs de $X$, $D$ et~$Z$ tels que $\mathscr{D}$ et $\mathscr{Z}$ soient des fermés de $\mathscr{X}$. En notant $\mathscr{U}=\mathscr{X}\backslash(\mathscr{D}\cup\mathscr{Z})$ on peut, quitte à agrandir $S$, écrire $\xi$ comme une famille de points locaux $(P_v)_{v\in\Omega_K}$ telle que $P_v\in \mathscr{U}(\mathscr{O}_v)$ pour $v\not\in S$.

Comme une fibre générale de $ev_{0,1}\vert_W$ est géométriquement intègre et rationnellement connexe, il suit de \cite[Theorem~1.2]{MR1937199} combiné à \cite[Theorem 2.13]{MR1158625} que les points rationnels sont denses dans une fibre de $ev_{0,1}\vert_W$ au-dessus d'un point rationnel général de $X$. Puisque $X$ vérifie l'approximation faible, on peut trouver $Q\in U(K)$ qui est arbitrairement proche des points $P_v$ pour $v\in S$. Choisissons~$S'$ fini contenant $S$ tel que $Q$ est entier hors de $S'$. L'approximation faible sur $X$ assure également qu'il existe $Q'\in U(K)$ tel que $Q'$ soit arbitrairement proche des $P_v$ pour $v\in S'\backslash S$. Par le théorème des fonctions implicites, la paire $(Q,Q')$ peut être choisie dans tout ouvert de~$X\times X$. On peut alors supposer que les points rationnels de la fibre de~$ev_{0,1}\vert_W$ au-dessus de $(Q,Q')$ y sont denses et il existe donc $f:\mathbf{P}^1_K\rightarrow X$ telle que $f(0)=Q$, $f(1)=Q'$ et $\mathbf{A}^1_K\subset f^{-1}(X\backslash Z)$. En outre, la scholie \ref{scholie} assure que cette dernière inclusion s'écrit également $\mathbf{A}^1_K\subset f^{-1}(X\backslash(D\cup Z))$. Choisissons alors un ensemble fini de places $S''$ de $K$ contenant~$S'$ tel que $f\vert_{\mathbf{A}^1_K}$ s'étende en un morphisme $f_{S''}:\mathbf{A}^1_{\mathscr{O}_{K,S''}}\rightarrow\mathscr{U}\otimes_{\mathscr{O}_{K,S}}\mathscr{O}_{K,S''}$.

Puisque $\mathbf{A}^1_K$ vérifie l'approximation forte hors de $v_0$, il existe $x\in\mathbf{A}^1(K)$ arbitrairement proche de $0$ aux places de $(S''\backslash S')\cup(S\backslash\{v_0\})$, de $1$ aux places de $S'\backslash S$ et appartenant à~$\mathbf{A}^1(\mathscr{O}_v)$ pour $v$ hors de $S''$. Mais alors, le point $f(x)$ est un élément de $U(K)$ arbitrairement proche de $Q$ aux places de $S\backslash\{v_0\}$ et entier hors de $S$: en effet, d'une part $Q'$ est entier aux places de $S'\backslash S$ donc $f(x)$ également; d'autre part le point $Q$ est entier en les places de $S''\backslash S'$ donc $f(x)$ l'est également; enfin $x$ est entier hors de $S''$ donc $f(x)=f_{S''}(x)$ est entier hors de $S''$. Puisque $Q$ est pris proche des $P_v$ pour $v\in S\backslash\{v_0\}$, on en déduit que~$f(x)$ convient.
\end{proof}

La proposition suivante montre que si $Z$ est de codimension au moins $2$ dans $X$, la condition (iii) de la proposition précédente est impliquée par l'hypothèse (ii):

\begin{prop}\label{ouvcompirr}
Si $Z$ est un fermé de $X$ tel que $\codim(Z,X)\geq2$ et si $F$ est une composante irréductible de $H$ telle que $\mathbf{P}^1_K\times F$ domine $X$ via $ev$, alors il existe un ouvert non vide $W$ de $F$ tel que $ev(\mathbf{A}^1_K\times W)\subset X\backslash Z$.
\end{prop}
\begin{proof}
La question étant géométrique, on peut aussitôt se ramener au cas où $K$ est algébriquement clos et raisonner en termes de points.

Comme $H$ est localement noethérien, choisissons $F^o$ un ouvert non vide de $F$ qui ne rencontre aucune autre composante irréductible de $H$. Notons $ev^o$ la restriction de~$ev$ à~$\mathbf{P}^1_K\times F^o$. Le morphisme $ev^o$ étant dominant, il existe par \cite[Chapter III, Proposition~10.6]{MR0463157} un ouvert non vide $U$ de $X$ tel que $ev^o$ soit de différentielle surjective en tout point de $(ev^o)^{-1}(U)$. La projection sur $F^o$ de $(ev^o)^{-1}(U)$ est alors un ouvert non vide de~$F^o$, qui est lisse en vertu du corollaire \ref{corcrucial}. Quitte à réduire $F^o$, on peut donc supposer que $F^o$ est lisse et que $(ev^o)^{-1}(U)$ se surjecte sur $F^o$. Mais le lemme \ref{lemkollar} assure alors que la différentielle de $ev^o$ est surjective en tout point de $\mathbf{A}^1_K\times F^o$. On déduit alors de \cite[Chapter III, Proposition 10.4]{MR0463157} que la restriction $ev^{oo}$ de $ev^o$ à~$\mathbf{A}^1_K\times F^o$ est lisse et donc plate.

Comme $Z$ est un fermé de codimension au moins $2$ de $X$, la platitude de $ev^{oo}$ assure que~$(ev^{oo})^{-1}(Z)$ est un fermé de codimension au moins $2$ de $\mathbf{A}^1_K\times F^o$. L'adhérence de sa projection sur~$F^o$ est alors un fermé strict de $F^o$ dont on note $W$ le complémentaire. Il vient donc que la restriction de $ev^{oo}$ à $\mathbf{A}^1_K\times W$ est un morphisme dominant dont l'image est dans~$X\backslash Z$.
\end{proof}
La démonstration du théorème \ref{purafmod} s'en déduit aussitôt:
\begin{proof}[Démonstration du théorème \ref{purafmod}]
Soit $Z$ un fermé de codimension $2$ de $X$. Il s'agit d'appli\-quer le proposition \ref{afmod} et donc de vérifier que la condition (iii) est satisfaite. Or, par (ii), le morphisme $ev_{0,1}\vert_F$ est dominant, ce qui implique que la restriction de $ev$ à $\mathbf{P}^1_K\times F$ l'est aussi. Le résultat découle donc de la proposition \ref{ouvcompirr}.
\end{proof}

\subsection{Intersections complètes affines lisses de bas multidegré}\label{total}

Dans cette section $X$ désigne une intersection complète lisse dans $\textbf{P}^n_K$ de type $(d_1,\dots,d_c)$ telle que $\sum_{i=1}^cd_i^2\leq n$. On choisit $D$ une section hyperplane lisse de $X$ et on reprend les notations de la sous-section précédente. Cette sous-section est dédiée au théorème \ref{thpuricl}:

\thpuricl*

\begin{proof}
Commençons par écarter le cas où $X\backslash D$ est l'espace affine, auquel cas le résultat découle de la proposition \ref{affAf}. Pour démontrer le théorème, il s'agit alors de montrer que les hypothèses du théorème \ref{purafmod} sont vérifiées pour $X$.  D'une part, la variété~$X$ vérifie l'approximation faible par les travaux de de Jong et Starr dans \cite[Theorem~1.1]{dJS2006} (voir également \cite[Corollary 4.8]{MR2931861}).\\

D'autre part, l'hypothèse (ii) du théorème \ref{purafmod} est essentiellement \cite[Theorem~1.8]{zbMATH06921639} mais il nous faut être plus précis car les espaces de modules considérés par les auteurs diffèrent de ceux définis ici. Introduisons la notation suivante: si $\mathcal{S}$ est un log-schéma (resp. un log-champ: voir \cite[section 5]{MR2032986} pour la définition d'un log-champ), notons $\underline{\mathcal{S}}$ le schéma sous-jacent (resp. le champ sous-jacent): si cette log-structure est triviale, on notera indifféremment $\mathcal{S}$ ou $\underline{\mathcal{S}}$. De même, si $u:\mathcal{S}\rightarrow \mathcal{T}$ est un morphisme de log-schémas (resp. de log-champs) on notera $\underline{u}$ le morphisme de schémas (resp. de champs) sous-jacent.

Notons $d$ la dimension de $X$, $l$ un nombre premier et choisissons $\beta$ dans $$H^{2d-2}_{\etale}(X_{\overline{K}},\mathbf{Z}_l(d-1))$$ la classe d'une courbe. Considérons la catégorie $\mathcal{C}$ constituée des log-applications stables vers le log-schéma divisoriel~$(X,D)$, de source une log-courbe dont les fibres sont de classe~$\beta$ et rencontrant $D$ en un point marqué, avec deux points marqués supplémentaires d'ordre de contact trivial:~on pourra se référer à \cite[Appendix B]{MR3224717}, \cite[\S2]{MR3257836} et \cite[\S1.3]{zbMATH06921639} pour la définition précise de~$\mathcal{C}$ et à \cite[\S2.2]{MR4026447} pour la définition d'application log-stable et d'une log-courbe.

La catégorie $\mathcal{C}$ est naturellement fibrée au-dessus de la catégorie des log-schémas sur~$K$ et est en fait un log-champ $\mathscr{A}_2(X,\beta)$ (voir \cite[\S2.2]{MR4026447} ainsi que \cite[\S1.3]{zbMATH06921639}). Considérons donc la famille universelle:
\begin{center}
\begin{tikzcd}
\mathscr{C} \arrow[r] \arrow[d] & (X,D) \\
\mathscr{A}_2(X,\beta) & 
\end{tikzcd}
\end{center}qui est munie de trois sections $q,r,s:\mathscr{A}_2(X,\beta)\rightarrow \mathscr{C}$ de $\mathscr{C}\rightarrow\mathscr{A}_2(X,\beta)$ et où $\mathscr{C}\rightarrow(X,D)$ est une application log-stable. On dispose également d'un morphisme d'évaluation $$\widetilde{ev}:\mathscr{A}_2(X,\beta)\rightarrow X\times X$$ défini en envoyant $f$ sur $(f(q),f(r))$. Si on note $\alpha\in H^{2d-2}_{\etale}\left(X_{\overline{K}},\mathbf{Z}_l(d-1)\right)$ la classe d'une droite dans $X$ et si on pose $\beta=2\alpha$, Chen et Zhu montrent dans \cite[Theorem 1.8]{zbMATH06921639}, sous l'hypothèse $\sum_{i=1}^nd_i^2\leq n$, qu'une fibre générale de $$\underline{\widetilde{ev}}:\underline{\mathscr{A}_2(X,2\alpha)}\rightarrow\underline{X}\times\underline{X}$$ est un schéma géométriquement intègre et rationnellement connexe dont un point général paramètre une courbe lisse.

Choisissons donc un ouvert $U$ de~$\mathscr{A}_2(X,2\alpha)$ paramétrant des courbes lisses et tel que le morphisme $\underline{\widetilde{ev}}\vert_{\underline{U}}:\underline{U}\rightarrow\underline{X}\times\underline{X}$ est de fibre générale géométriquement intègre et rationnellement connexe. Quitte à restreindre~$U$, on peut également supposer qu'on a un isomorphisme $\mathscr{C}\vert_{U}\xrightarrow{\sim}\mathbf{P}^1\times U$ envoyant~$q$ (resp. $r$, resp. $s$) sur $0$ (resp. $1$, resp. $\infty$). En munissant $H$ de sa log-structure triviale, la propriété universelle de $H$ donne alors un log-morphisme $v:U\rightarrow H$. De plus, on dispose du $H$-point suivant de $\mathscr{A}_2(X,2\alpha)$:
\begin{center}
\begin{tikzcd}
H\times\mathbf{P}^1 \arrow[r, "ev"] \arrow[d] & (X,D) \\
H & 
\end{tikzcd}
\end{center}
où $H\times\mathbf{P}^1$ est muni de sa log-structure divisorielle donnée par le diviseur $H\times\{\infty\}$. Ceci donne donc un log-morphisme $w:H\rightarrow\mathscr{A}_2(X,2\alpha)$ et on vérifie que $\underline{v}$ et $\underline{w}$ sont des inverses birationnels l'un de l'autre.

Enfin, la construction de $v$ et $w$ assure que le diagramme suivant commute:
\begin{center}
\begin{tikzcd}
U \arrow[rr, "v"] \arrow[rd, "\widetilde{ev}\vert_{U}", swap] &  & H\arrow[ld, "ev_{0,1}"] \\
 & X\times X &
\end{tikzcd}.
\end{center}
Puisque $\underline{v}$ est un morphisme birationnel, on en déduit que les fibres génériques de $\underline{\widetilde{ev}}$ et de $ev_{0,1}$ sont birationnelles. Il s'ensuit qu'une fibre générale de $ev_{0,1}$ est géométriquement intègre et rationnellement connexe, ce qui assure que l'hypothèse (ii) du théorème \ref{purafmod} est vérifiée.
\end{proof}

\section*{Remerciements}

Mes plus vifs remerciements vont à Olivier Wittenberg, qui m'a introduit à cette question et me partagea généreusement ses idées pour l'aborder: les échanges mathématiques que nous eûmes, sa patience, ainsi que ses conseils avisés, me furent d'une aide précieuse tout au long de ce travail. Je suis également reconnaissant à Jason Starr de m'avoir partagé un argument technique permettant d'établir une version simplifiée du critère de pureté~\ref{purafmod} où n'apparaissait pas de section hyperplane. Je remercie enfin Oussama Hamza pour ses commentaires sur l'ensemble du texte.

\bibliographystyle{amsalpha-fr}
\bibliography{biblio}

\textsc{Institut Galilée, Université Sorbonne Paris Nord, 99~avenue Jean-Baptiste Clément, 93430 Villetaneuse, France}

\textit{Adresse électronique}: \href{mailto:boughattas@math.univ-paris13.fr}{boughattas@math.univ-paris13.fr}

\end{document}